\documentclass[12pt, a4paper, abstracton, bibliography=totoc]{scrartcl}
\pdfoutput=1

\usepackage{amsmath, amsthm, amsfonts, amssymb}
\usepackage[utf8]{inputenc}
\usepackage[T1]{fontenc}
\usepackage{color}
\usepackage{slashed} % Für die durchgestrichenen Dirac-Operator-Zeichen.
\usepackage[all, cmtip]{xy}
\usepackage{hyperref} % Anklickbare Links und Inhaltsverzeichnis. Und bookmarks im pdf.
\let\counterwithout\relax

\usepackage{chngcntr} % Um die Fußnotennummerierung fortlaufend im ganzen Dokument machen zu koennen.

\usepackage{graphicx} % Damit du größere \cdot definieren kannst. Und zum Bilder einbinden.
\usepackage[english]{babel}
\usepackage{microtype}
\usepackage{bm} % Für Fettschrift im Mathemodus.
\usepackage{mathtools} % Damit du den \mathclap-Befehl nutzen kannst.

\newcommand{\IN}{\mathbb{N}}
\newcommand{\IZ}{\mathbb{Z}}

\newcommand{\IR}{\mathbb{R}}
\newcommand{\IC}{\mathbb{C}}

\newcommand{\frakS}{\mathfrak{S}}
\newcommand{\injrad}{\operatorname{inj-rad}}
\newcommand{\Rm}{\operatorname{Rm}}

\newcommand{\LLip}{L\text{-}\operatorname{Lip}}
\newcommand{\IB}{\mathfrak{B}}
\newcommand{\IK}{\mathfrak{K}}

\newcommand{\supp}{\operatorname{supp}}
\newcommand{\diam}{\operatorname{diam}}
\newcommand{\Hom}{\operatorname{Hom}}

\newcommand{\id}{\operatorname{id}}
\newcommand{\kernel}{\operatorname{ker}}

\newcommand{\image}{\operatorname{im}}

\newcommand{\card}{\#}

\newcommand{\trace}{\operatorname{tr}}
\newcommand{\ch}{\operatorname{ch}}
\newcommand{\vol}{\operatorname{vol}}
\newcommand{\ind}{\operatorname{ind}}

\newcommand{\op}{\mathrm{op}}

\newcommand{\HPucont}{H \! P_\mathrm{cont}}
\newcommand{\HCucont}{H \! C_\mathrm{cont}}
\newcommand{\HHucont}{H \! H_\mathrm{cont}}

\newcommand{\HudR}{H^{u, \mathrm{dR}}}

\newcommand{\HbdR}{H_{b, \mathrm{dR}}}
\newcommand{\HudRco}{H_{u, \mathrm{dR}}}

\newcommand{\Cucont}{C_{\mathrm{cont}}}
\newcommand{\Clucont}{C_{\lambda, \mathrm{cont}}}
\newcommand{\Winftyone}{W^{\infty, 1}}

\newcommand{\ev}{\mathrm{ev}}
\newcommand{\odd}{\mathrm{odd}}
\newcommand{\UPsiDO}{\mathrm{U}\Psi\mathrm{DO}}

\newcommand{\Frechet}{Fr\'{e}chet }
\newcommand{\Poincare}{Poincar\'{e} }
\newcommand{\Folner}{F{\o}lner }
\newcommand{\Spakula}{\v{S}pakula }
\newcommand{\spinc}{spin$^c$ }

\newcommand*{\largecdot}{\raisebox{-0.25ex}{\scalebox{1.2}{$\cdot$}}}

\DeclareMathOperator{\hatotimes}{\hat{\otimes}}
\DeclareMathOperator{\barotimes}{\bar{\otimes}}

\theoremstyle{plain}
\newtheorem{thm}{Theorem}[section]
\newtheorem*{thm*}{Theorem}
\newtheorem*{mainthm*}{Main Result}
\newtheorem{cor}[thm]{Corollary}
\newtheorem*{cor*}{Corollary}
\newtheorem{lem}[thm]{Lemma}
\newtheorem{prop}[thm]{Proposition}

\newtheorem{question}[thm]{Question}

\newtheorem{thmintro}{Theorem}

\newtheorem{corintro}[thmintro]{Corollary}

\theoremstyle{definition}
\newtheorem{defn-alt}[thm]{Definition}
\newtheorem{example-alt}[thm]{Example}
\newtheorem{examples-alt}[thm]{Examples}
\newtheorem{rem-alt}[thm]{Remark}
\newtheorem{nota-alt}[thm]{Notation}

\newenvironment{defn}    
{
	\pushQED{\qed}\begin{defn-alt}}
	{\popQED\end{defn-alt}}

\newenvironment{example}    
{
	\pushQED{\qed}\begin{example-alt}}
	{\popQED\end{example-alt}}
	
\newenvironment{examples}    
{
	\pushQED{\qed}\begin{examples-alt}}
	{\popQED\end{examples-alt}}

\newenvironment{rem}    
{
	\pushQED{\qed}\begin{rem-alt}}
	{\popQED\end{rem-alt}}

\numberwithin{equation}{section}

\counterwithout{footnote}{section}

\makeatletter% Damit man Footnotes ohne Nummer machen kann.
\def\blfootnote{\gdef\@thefnmark{}\@footnotetext}
\makeatother

\begin{document}

\title{Index theorems for\\ uniformly elliptic operators}
\author{Alexander Engel}
\date{}
\maketitle

\vspace*{-3.5\baselineskip}
\begin{center}
\footnotesize{
\textit{
Fakult\"{a}t f\"{u}r Mathematik\\
Universit\"{a}t Regensburg\\
93040 Regensburg, GERMANY\\
\href{mailto:alexander.engel@mathematik.uni-regensburg.de}{alexander.engel@mathematik.uni-regensburg.de}
}}
\end{center}

\vspace*{-0.5\baselineskip}
\begin{abstract}
%\blfootnote{\textit{Date:} \today.}
%\blfootnote{\textit{$2010$ Mathematics Subject Classification.} Primary:\ 58J20; Secondary:\ 19K56, 47G30.}
%\blfootnote{\textit{Keywords and phrases.} Index theory, pseudodifferential operators, uniform $K$-homology.}
We generalize Roe's index theorem for graded generalized Dirac operators on amenable manifolds to multigraded elliptic uniform pseudodifferential operators.

The generalization will follow from a local index theorem that is valid on any manifold of bounded geometry. This local formula incorporates the uniform estimates present in the definition of uniform pseudodifferential operators.
\end{abstract}

\tableofcontents

\section{Introduction}

Recall the following index theorem of Roe for amenable manifolds (with notation adapted to the one used in this article):

\begin{thm*}[{\cite[Theorem 8.2]{roe_index_1}}]
Let $M$ be a Riemannian manifold of bounded geometry and $D$ a generalized Dirac operator associated to a graded Dirac bundle $S$ of bounded geometry over $M$.

Let $(M_i)_i$ be a \Folner sequence\footnote{That is to say, for every $r > 0$ we have $\frac{\vol B_r(\partial M_i)}{\vol M_i} \stackrel{i \to \infty}\longrightarrow 0$. Manifolds admitting such a sequence are called \emph{amenable}.} for $M$, $\tau \in (\ell^\infty)^\ast$ a linear functional associated to a free ultrafilter on $\IN$, and $\theta$ the corresponding trace on the uniform Roe algebra of $M$.

Then we have
\[\theta(\mu_u(D)) = \tau \Big( \frac{1}{\vol M_i} \int_{M_i} \ind(D) \Big).\]
\end{thm*}

Here $\ind(D)$ is the usual integrand for the topological index of $D$ in the Atiyah--Singer index formula, so the right hand side is topological in nature. On the left hand side of the formula we have the coarse index class $\mu_u(D) \in K_0(C_u^\ast(M))$ of $D$ in the $K$-theory of the uniform Roe algebra of $M$ evaluated under the trace $\theta$.  This is an analytic expression and may be computed as $\theta(\mu_u(D)) = \tau \Big( \frac{1}{\vol M_i} \int_{M_i} \trace_s k_{f(D)}(x,x) \ dx \Big)$, where $k_{f(D)}(x,y)$ is the integral kernel of the smoothing operator $f(D)$, where $f$ is an even Schwartz function with $f(0) = 1$.

In this article we will generalize this theorem to all multigraded, elliptic, symmetric uniform pseudodifferential operators. So especially we also encompass Toeplitz operators since they are included in the ungraded case. This generalization will follow from a local index theorem that will hold on any manifold of bounded geometry, i.e., without an amenability assumption on $M$.

Let us state our local index theorem in the formulation using twisted Dirac operators associated to \spinc structures:

\begin{thmintro}[Theorem~\ref{thm:local_thm_twisted}]
Let $M$ be an $m$-dimensional \spinc manifold of bounded geometry and without boundary. Denote the associated Dirac operator by $D$.

Then we have the following commutative diagram:
\[\xymatrix{
K^\ast_u(M) \ar[r]^-{- \cap [D]}_-{\cong} \ar[d]_-{\ch(-) \wedge \ind(D)} & K_{m-\ast}^u(M) \ar[d]^-{\alpha_\ast \circ \ch^\ast}\\
\HbdR^\ast(M) \ar[r]_-{\cong} & \HudR_{m-\ast}(M)}\]
where in the top row $\ast$ is either $0$ or $1$ and in the bottom row $\ast$ is either $\ev$ or $\odd$.
\end{thmintro}

Here $K_{m-\ast}^u(M)$ is uniform $K$-homology of $M$ invented by \Spakula \cite{spakula_uniform_k_homology} and $K_u^\ast(M)$ is the corresponding uniform $K$-theory which we will recall in Section~\ref{sec_uniform_K}. The map $- \cap [D]$ is the cap product and that it is an isomorphism was shown in \cite[Section~4.4]{engel_indices_UPDO}. Moreover, $\HbdR^\ast(M)$ denotes the bounded de Rham cohomology of $M$ and $\ind(D)$ the topological index class of $D$ in there. Furthermore, $\HudR_{m-\ast}(M)$ is the uniform de Rham homology of $M$ to be defined in Section \ref{sec:u_de_Rham_currents} via Connes' cyclic cohomology, and that it is \Poincare dual to bounded de Rham cohomology is proved in Theorem \ref{thm:Poincare_de_Rham}. Finally, let us note that we will also prove in Section \ref{sec:chern_isos} that the Chern characters induce isomorphisms after a certain completion that also kills torsion, similar to the case of compact manifolds.

Using a series of steps as in Connes' and Moscovici's proof of \cite[Theorem 3.9]{connes_moscovici} we will generalize the above computation of the \Poincare dual of $(\alpha_\ast \circ \ch^\ast)([D]) \in \HudR_{m-\ast}(M)$ to symmetric and elliptic uniform pseudodifferential operators:

\begin{thmintro}[Theorem~\ref{thm:local_thm_pseudos} and Remark~\ref{rem_finitely_summable_not_needed}]
Let $M$ be an oriented Riemannian manifold of bounded geometry and without boundary, and $P$ be a symmetric and elliptic uniform pseudodifferential operator of positive order.

Then $\ind(P) \in \HbdR^\ast(M)$ is the \Poincare dual of $(\alpha_\ast \circ \ch^\ast)([P]) \in \HudR_\ast(M)$.
\end{thmintro}

Using the above local index theorem we will derive as a corollary the following local index formula:

\begin{corintro}[Corollary~\ref{cor:pairing_compactly_supported}]
Let $[\varphi] \in H_{c, \mathrm{dR}}^k(M)$ be a compactly supported cohomology class and define the analytic index $\ind_{[\varphi]}(P)$ as Connes--Moscovici \cite{connes_moscovici} for $P$ being a multigraded, symmetric, elliptic uniform pseudodifferential operator of positive order. Then we have
\[\ind_{[\varphi]}(P) = \int_M \ind(P) \wedge [\varphi]\]
and this pairing is continuous, i.e., $\int_M \ind(P) \wedge [\varphi] \le \|\ind(P) \|_\infty \cdot \| [\varphi] \|_1$, where $\| - \|_\infty$ denotes the sup-seminorm on $\HbdR^{m-k}(M)$ and $\| - \|_1$ the $L^1$-seminorm on $H_{c, \mathrm{dR}}^k(M)$.
\end{corintro}

Note that the corollary reads basically the same as the local index formula of Connes and Moscovici \cite{connes_moscovici}. The fundamentally new thing in it is the continuity statement for which we need the uniformity assumption for $P$.

As a second corollary to the above local index theorem we will, as already written, derive the generalization of Roe's index theorem for amenable manifolds.

\begin{corintro}[Corollary~\ref{cor:pairing_global}]
Let $M$ be a manifold of bounded geometry and without boundary, let $(M_i)_i$ be a \Folner sequence for $M$ and let $\tau \in (\ell^\infty)^\ast$ be a linear functional associated to a free ultrafilter on $\IN$. Denote the from the choice of \Folner sequence and functional $\tau$ resulting functional on $K_0(C_u^\ast(M))$ by $\theta$.

Then for both $p \in \{ 0,1 \}$, every class $[P] \in K_p^u(M)$ with $P$ being a $p$-graded, symmetric, elliptic uniform pseudodifferential operator over $M$, and every $u \in K_u^p(M)$ we have
\[\langle u, [P] \rangle_\theta = \langle \ch(u) \wedge \ind(P), [M] \rangle_{(M_i)_i, \tau}.\]
\end{corintro}

Roe's theorem \cite{roe_index_1} is the special case where $P = D$ is a graded (i.e., $p = 0$) Dirac operator and $u = [\IC]$ is the class in $K_u^0(M)$ of the trivial, $1$-dimensional vector bundle over $M$.

To put the above index theorems into context, let us consider manifolds with cylindrical ends. These are the kind of non-compact manifolds which are studied to prove for example the Atiyah--Patodi--Singer index theorem. In the setting of this paper, the relevant algebra would be that of bounded functions with bounded derivatives, whereas in papers like \cite{melrose_eta} or \cite{moroianu_nistor} one imposes conditions at infinity like rapid decay of the integral kernels (see the definition of the suspended algebra in \cite[Section~1]{melrose_eta}).

Note that this global index theorem arising from a \Folner sequence is just a special case of a certain rough index theory, where one pairs classes from the so-called rough cohomology with classes in the $K$-theory of the uniform Roe algebra, and \Folner sequences give naturally classes in this rough cohomology. For details see the thesis \cite{mavra} of Mavra. It seems that it should be possible to combine the above local index theorem with this rough index theory, since it is possible in the special case of \Folner sequences. The author investigated this in \cite{engel_rough}.

Let us say a few words about the proofs of the above index theorems for elliptic uniform pseudodifferential operators. Roe used in \cite{roe_index_1} the heat kernel method to prove his index theorem for amenable manifolds and therefore, since the heat kernel method does only work for Dirac operators, it can not encompass uniform pseudodifferential operators. So what we will basically do in this paper is to set up all the necessary theory in order to be able to reduce the index problem from pseudodifferential operators to Dirac operators.

The main ingredient is a version of \Poincare duality between uniform $K$-homology and uniform $K$-theory, which was proved by the author in \cite[Section~4.4]{engel_indices_UPDO}. With this at our disposal we will then be able to reduce the index problem for elliptic uniform pseudodifferential operators to Dirac operators by proving a uniform version of the Thom isomorphism in order to conclude that symbol classes of elliptic uniform pseudodifferential operators may be represented by symbol classes of Dirac operators. So it remains to show the local index theorem for Dirac operators, but since up to this point we will already have set up all the needed machinery, this proof will be basically the same as the proof of the local index theorem of Connes and Moscovici in \cite{connes_moscovici}.

The last collection of results that we want to highlight in this introduction are all the various (duality) isomorphisms proved in this paper.

\begin{thmintro}[Theorems~\ref{thm23w2}, \ref{thm:computation_cyclic_cohomology} and \ref{thm:Poincare_de_Rham}]
Let $M$ be an $m$-dimensional manifold of bounded geometry and no boundary. Then the Chern characters induce linear, continuous isomorphisms
\[K^\ast_u(M) \barotimes \IC \cong \HbdR^\ast(M) \text{ and } K^u_\ast(M) \barotimes \IC \cong \HudR_\ast(M),\]
and we also have the isomorphism
\[\HPucont^{*}(\Winftyone(M)) \cong \HudR_{*}(M).\]
If $M$ is oriented we further have the isomorphism
\[\HbdR^\ast(M) \cong \HudR_{m-\ast}(M).\]
\end{thmintro}

If $M$ is \spinc then we have the \Poincare duality isomorphism $K_u^\ast(M) \cong K^\ast_{m-\ast}(M)$, which is proved in \cite[Theorem 4.29]{engel_indices_UPDO}.

\paragraph{Acknowledgements} This article contains mostly Section~5 of the preprint \cite{engel_indices_UPDO} which is being split up for easier publication. It arose out of the Ph.D.\ thesis \cite{engel_phd} of the author written at the University of Augsburg.

\section{Review of needed material}

In this section we review the needed material from the literature. We start with the notion of bounded geometry for Riemannian manifolds, define Sobolev spaces and discuss the Sobolev embedding theorem, and at the end of Section~\ref{sec_bounded_geom} we prove the technical Lemma \ref{lem:suitable_coloring_cover_M} about constructing covers with certain properties on manifolds of bounded geometry. In Section~\ref{sec_UPDOs} we discuss the calculus of uniform pseudodifferential operators that we will use in this paper, and in Section~\ref{sec_uniform_K} we recall the basic facts about uniform $K$-homology and uniform $K$-theory.

\subsection{Manifolds of bounded geometry}
\label{sec_bounded_geom}

We will recall in this section the notion of bounded geometry for manifolds and for vector bundles and discuss basic facts about uniform $C^r$-spaces and Sobolev spaces on them. Almost all material presented here is already known, and we tried to give proper credits wherever possible. As a genuine reference one might also use Eldering \cite[Chapter 2]{eldering_book}.

\begin{defn}
We say that a Riemannian manifold $M$ has \emph{bounded geometry}, if
\begin{itemize}
\item the curvature tensor and all its derivatives are bounded, i.e.,
\[\sup_{x \in M} \| \nabla^k \Rm (x) \| < \infty\]
for all $k \in \IN_0$, and
\item the injectivity radius is uniformly positive, i.e.,
\[\inf_{x \in M} \injrad_M(x) > 0.\]
\end{itemize}
If $E \to M$ is a vector bundle with a metric and compatible connection, then \emph{$E$ has bounded geometry}, if the curvature tensor of $E$ and all its derivatives are bounded.
\end{defn}

\begin{examples}
The most important examples of manifolds of bounded geometry are coverings of closed Riemannian manifolds equipped with the pull-back metric, homogeneous manifolds with an invariant metric, and leafs in a foliation of a compact Riemannian manifold (Greene \cite[lemma on page 91 and the paragraph thereafter]{greene}).

For vector bundles, the most important examples are of course again pull-back bundles of bundles over closed manifolds equipped with the pull-back metric and connection, and the tangent bundle of a manifold of bounded geometry.
\end{examples}

We now state an important characterization in local coordinates of bounded geometry since it allows one to show that certain local definitions are independent of the chosen normal coordinates.

\begin{lem}[{\cite[Appendix A1.1]{shubin}}]\label{lem:transition_functions_uniformly_bounded}
Let the injectivity radius of $M$ be positive.

Then the curvature tensor of $M$ and all its derivatives are bounded if and only if for any $0 < r < \injrad_M$ all the transition functions between overlapping normal coordinate charts of radius $r$ are uniformly bounded, as are all their derivatives (i.e., the bounds can be chosen to be the same for all transition functions).
\end{lem}

Another fact which we will need about manifolds of bounded geometry is the existence of uniform covers by normal coordinate charts and corresponding partitions of unity. A proof may be found in, e.g., \cite[Appendix A1.1]{shubin} (Shubin addresses the first statement about the existence of such covers actually to the paper \cite{gromov_curvature_diameter_betti_numbers} of Gromov).

\begin{lem}\label{lem:nice_coverings_partitions_of_unity}
Let $M$ be a manifold of bounded geometry.

For every $0 < \varepsilon < \tfrac{\injrad_M}{3}$ exists a cover of $M$ by normal coordinate charts of radius $\varepsilon$ with the properties that the midpoints of the charts form a uniformly discrete set and that the coordinate charts with double radius $2\varepsilon$ form a uniformly locally finite cover of $M$.

Furthermore, there is a subordinate partition of unity $1 = \sum_i \varphi_i$ with $\supp \varphi_i \subset B_{2\varepsilon}(x_i)$, such that in normal coordinates the functions $\varphi_i$ and all their derivatives are uniformly bounded (i.e., the bounds do not depend on $i$).
\end{lem}

If the manifold $M$ has bounded geometry, we have analogous equivalent local characterizations of bounded geometry for vector bundles as for manifolds. The equivalence of the first two bullet points in the next lemma is stated in, e.g., \cite[Proposition~2.5]{roe_index_1}. Concerning the third bullet point, the author could not find any citable reference in the literature (though both Shubin \cite{shubin} and Eldering \cite{eldering_book} use this as the definition).

\begin{lem}\label{lem:equiv_characterizations_bounded_geom_bundles}
Let $M$ be a manifold of bounded geometry and $E \to M$ a vector bundle. Then the following are equivalent:

\begin{itemize}
\item $E$ has bounded geometry,
\item the Christoffel symbols $\Gamma_{i \alpha}^\beta(y)$ of $E$ with respect to synchronous framings (considered as functions on the domain $B$ of normal coordinates at all points) are bounded, as are all their derivatives, and this bounds are independent of $x \in M$, $y \in \exp_x(B)$ and $i, \alpha, \beta$, and
\item the matrix transition functions between overlapping synchronous framings are uniformly bounded, as are all their derivatives (i.e., the bounds are the same for all transition functions).
\end{itemize}
\end{lem}

We will now give the definition of uniform $C^\infty$-spaces together with a local characterization on manifolds of bounded geometry. The interested reader is refered to, e.g., the papers \cite[Section 2]{roe_index_1} or \cite[Appendix A1.1]{shubin} of Roe and Shubin for more information regarding these uniform $C^\infty$-spaces.

\begin{defn}[$C^r$-bounded functions]
Let $f \in C^\infty(M)$. We say that $f$ is a \emph{$C_b^r$-function}, or equivalently that it is \emph{$C^r$-bounded}, if $\| \nabla^i f \|_\infty < C_i$ for all $0 \le i \le r$.
\end{defn}

If $M$ has bounded geometry, being $C^r$-bounded is equivalent to the statement that in every normal coordinate chart $|\partial^\alpha f(y)| < C_\alpha$ for every multiindex $\alpha$ with $|\alpha| \le r$ (where the constants $C_\alpha$ are independent of the chart).

The definition of $C^r$-boundedness and its equivalent characterization in normal coordinate charts for manifolds of bounded geometry make also sense for sections of vector bundles of bounded geometry.

\begin{defn}[Uniform $C^\infty$-spaces]
\label{defn:uniform_frechet_spaces}
Let $E$ be a vector bundle of bounded geometry over $M$. We will denote the \emph{uniform $C^r$-space} of all $C^r$-bounded sections of $E$ by $C_b^r(E)$.

Furthermore, we define the \emph{uniform $C^\infty$-space $C_b^\infty(E)$}
\[C_b^\infty(E) := \bigcap_r C_b^r(E)\]
which is a \Frechet space.
\end{defn}

Now we get to Sobolev spaces on manifolds of bounded geometry. Much of the following material is from \cite[Appendix A1.1]{shubin} and \cite[Section 2]{roe_index_1}, where an interested reader can find more thorough discussions of this matters.

Let $s \in C^\infty_c(E)$ be a compactly supported, smooth section of some vector bundle $E \to M$ with metric and connection $\nabla$. For $k \in \IN_0$ and $p \in [1,\infty)$ we define the global $W^{k,p}$-Sobolev norm of $s$ by
\begin{equation}\label{eq:sobolev_norm}
\|s\|_{W^{k,p}}^p := \sum_{i=0}^k \int_M \|\nabla^i s(x)\|^p dx.
\end{equation}

\begin{defn}[Sobolev spaces $W^{k,p}(E)$]\label{defn:sobolev_spaces}
Let $E$ be a vector bundle which is equipped with a metric and a connection. The \emph{$W^{k,p}$-Sobolev space of $E$} is the completion of $C^\infty_c(E)$ in the norm $\|-\|_{W^{k,p}}$ and will be denoted by $W^{k,p}(E)$.
\end{defn}

If $E$ and $M^m$ both have bounded geometry than the Sobolev norm \eqref{eq:sobolev_norm} for $1 < p < \infty$ is equivalent to the local one given by
\begin{equation}\label{eq:sobolev_norm_local}
\|s\|_{W^{k,p}}^p \stackrel{\text{equiv}}= \sum_{i=1}^\infty \|\varphi_i s\|^p_{W^{k,p}(B_{2\varepsilon}(x_i))},
\end{equation}
where the balls $B_{2\varepsilon}(x_i)$ and the subordinate partition of unity $\varphi_i$ are as in Lemma \ref{lem:nice_coverings_partitions_of_unity}, we have chosen synchronous framings and $\|-\|_{W^{k,p}(B_{2\varepsilon}(x_i))}$ denotes the usual Sobolev norm on $B_{2\varepsilon}(x_i) \subset \IR^m$. This equivalence enables us to define the Sobolev norms for all $k \in \IR$, see Triebel \cite{triebel_2} and Gro{\ss}e--Schneider \cite{grosse_sobolev}. There are some issues in the case $p=1$, see the discussion by Triebel \cite[Section~2.2.3]{triebel_1}, \cite[Remark~4 on Page~13]{triebel_2}.

Assuming bounded geometry, the usual embedding theorems are true:

\begin{thm}[{\cite[Theorem 2.21]{aubin_nonlinear_problems}}]\label{thm:sobolev_embedding}
Let $E$ be a vector bundle of bounded geometry over a manifold $M^m$ of bounded geometry and without boundary.

Then we have for all values $(k-r)/m > 1/p$ continuous embeddings
\[W^{k,p}(E) \subset C^r_b(E).\]
\end{thm}

We define the space
\begin{equation}
\label{eq:defn_W_infty}
W^{\infty,p}(E) := \bigcap_{k \in \IN_0} W^{k,p}(E)
\end{equation}
and equip it with the obvious \Frechet topology. The Sobolev Embedding Theorem tells us now that we have for all $p$ a continuous embedding
\[W^{\infty,p}(E) \hookrightarrow C^\infty_b(E).\]

Finally, we come to a technical statement (Lemma~\ref{lem:suitable_coloring_cover_M}) about the existence of open covers with special properties on manifolds of bounded geometry, similar to Lemma~\ref{lem:nice_coverings_partitions_of_unity}. As a preparation we first have to recall some facts about simplicial complexes of bounded geometry and corresponding triangulations of manifolds of bounded geometry.

\begin{defn}[Bounded geometry simplicial complexes]\label{defn:simplicial_complex_bounded_geometry}
A simplicial complex has \emph{bounded geometry} if there is a uniform bound on the number of simplices in the link of each vertex.

A subdivision of a simplicial complex of bounded geometry with the properties that
\begin{itemize}
\item each simplex is subdivided a uniformly bounded number of times on its $n$-skeleton, where the $n$-skeleton is the union of the $n$-dimensional sub-simplices of the simplex, and that
\item the distortion $\operatorname{length}(e) + \operatorname{length}(e)^{-1}$ of each edge $e$ of the subdivided complex is uniformly bounded in the metric given by barycentric coordinates of the original complex,
\end{itemize}
is called a \emph{uniform subdivision}.
\end{defn}

\begin{thm}[Attie {\cite[Theorem 1.14]{attie_classification}}]\label{thm:triangulation_bounded_geometry}
Let $M$ be a manifold of bounded geometry and without boundary.

Then $M$ has a triangulation as a simplicial complex of bounded geometry such that the metric given by barycentric coordinates is bi-Lipschitz equivalent\footnote{Two metric spaces $X$ and $Y$ are said to be \emph{bi-Lipschitz equivalent} if there is a homeomorphism $f\colon X \to Y$ with
\[\tfrac{1}{C} d_X(x,x^\prime) \le d_Y(f(x), f(x^\prime)) \le C d_X(x,x^\prime)\]
for all $x,x^\prime \in X$ and some constant $C > 0$.} to the one on $M$ induced by the Riemannian structure. This triangulation is unique up to uniform subdivision.

Conversely, if $M$ is a simplicial complex of bounded geometry which is a triangulation of a smooth manifold, then this smooth manifold admits a metric of bounded geometry with respect to which it is bi-Lipschitz equivalent to $M$.
\end{thm}

\begin{rem}\label{rem:attie_regularity}
Attie uses in \cite{attie_classification} a weaker notion of bounded geometry as we do: additionally to a uniformly positive injectivity radius he only requires the sectional curvatures to be bounded in absolute value (i.e., the curvature tensor is bounded in norm), but he assumes nothing about the derivatives (see \cite[Definition 1.4]{attie_classification}). But going into his proof of \cite[Theorem 1.14]{attie_classification}, we see that the Riemannian metric constructed for the second statement of the theorem is actually of bounded geometry in our strong sense (i.e., also with bounds on the derivatives of the curvature tensor).

As a corollary we get that for any manifold of bounded geometry in Attie's weak sense there is another Riemannian metric of bounded geometry in our strong sense that is bi-Lipschitz equivalent the original one (in fact, this bi-Lipschitz equivalence is just the identity map of the manifold, as can be seen from the proof).
\end{rem}

The last auxiliary lemma (before we come to the crucial Lemma~\ref{lem:suitable_coloring_cover_M}) is about coloring covers of manifolds with only finitely many colors:

\begin{lem}\label{lem:coloring_graph}
Let a covering $\{U_\alpha\}$ of $M$ with finite multiplicity be given. Then there exists a coloring of the subsets $U_\alpha$ with finitely many colors such that no two intersecting subsets have the same color.
\end{lem}

\begin{proof}
Construct a graph whose vertices are the subsets $U_\alpha$ and two vertices are connected by an edge if the corresponding subsets intersect. We have to find a coloring of this graph with only finitely many colors where connected vertices do have different colors.

To do this, we firstly use the theorem of de Bruijin--Erd\"{o}s stating that an infinite graph may be colored by $k$ colors if and only if every of its finite subgraphs may be colored by $k$ colors (one can use the Lemma of Zorn to prove this).

Secondly, since the covering has finite multiplicity it follows that the number of edges attached to each vertex in our graph is uniformly bounded from above, i.e., the maximum vertex degree of our graph is finite. But this also holds for every subgraph of our graph, with the maximum vertex degree possibly only decreasing by passing to a subgraph. Now a simple greedy algorithm shows that every finite graph may be colored with one more color than its maximum vertex degree: just start by coloring a vertex with some color, go to the next vertex and use an admissible color for it, and so on.
\end{proof}

\begin{lem}\label{lem:suitable_coloring_cover_M}
Let $M$ be a manifold of bounded geometry and without boundary.

Then there is an $\varepsilon > 0$ and a countable collection of uniformly discretely distributed points $\{x_i\}_{i \in I} \subset M$ such that $\{B_{\varepsilon}(x_i)\}_{i \in I}$ is a uniformly locally finite cover of $M$. We can additionally arrange such that it has the following two properties:
\begin{enumerate}
\item \label{cover_i} It is possible to partition $I$ into a finite amount of subsets $I_1, \ldots, I_N$ such that for each $1 \le j \le N$ the subset $U_j := \bigcup_{i \in I_j} B_{\varepsilon}(x_i)$ is a disjoint union of balls that are a uniform distance apart from each other, and such that for each $1 \le K \le N$ the connected components of $U_K := U_1 \cup \ldots \cup U_k$ are also a uniform distance apart from each other (see Figure~\ref{fig:not_allowed_cover}).
\item \label{cover_ii} Instead of choosing balls $B_\varepsilon(x_i)$ to get our cover of $M$ it is possible to choose other open subsets such that additionally to the property from Point~\ref{cover_i} for any distinct $1 \le m,n \le N$ the symmetric difference $U_m \Delta U_n$ consists of open subsets of $M$ which are a uniform distance apart from each other.\footnote{To see a non-example, in the lower part of Figure \ref{fig:not_allowed_cover} this is actually \emph{not} the case.}
\end{enumerate}
\end{lem}

\begin{figure}[ht]
\centering
\includegraphics[scale=0.5]{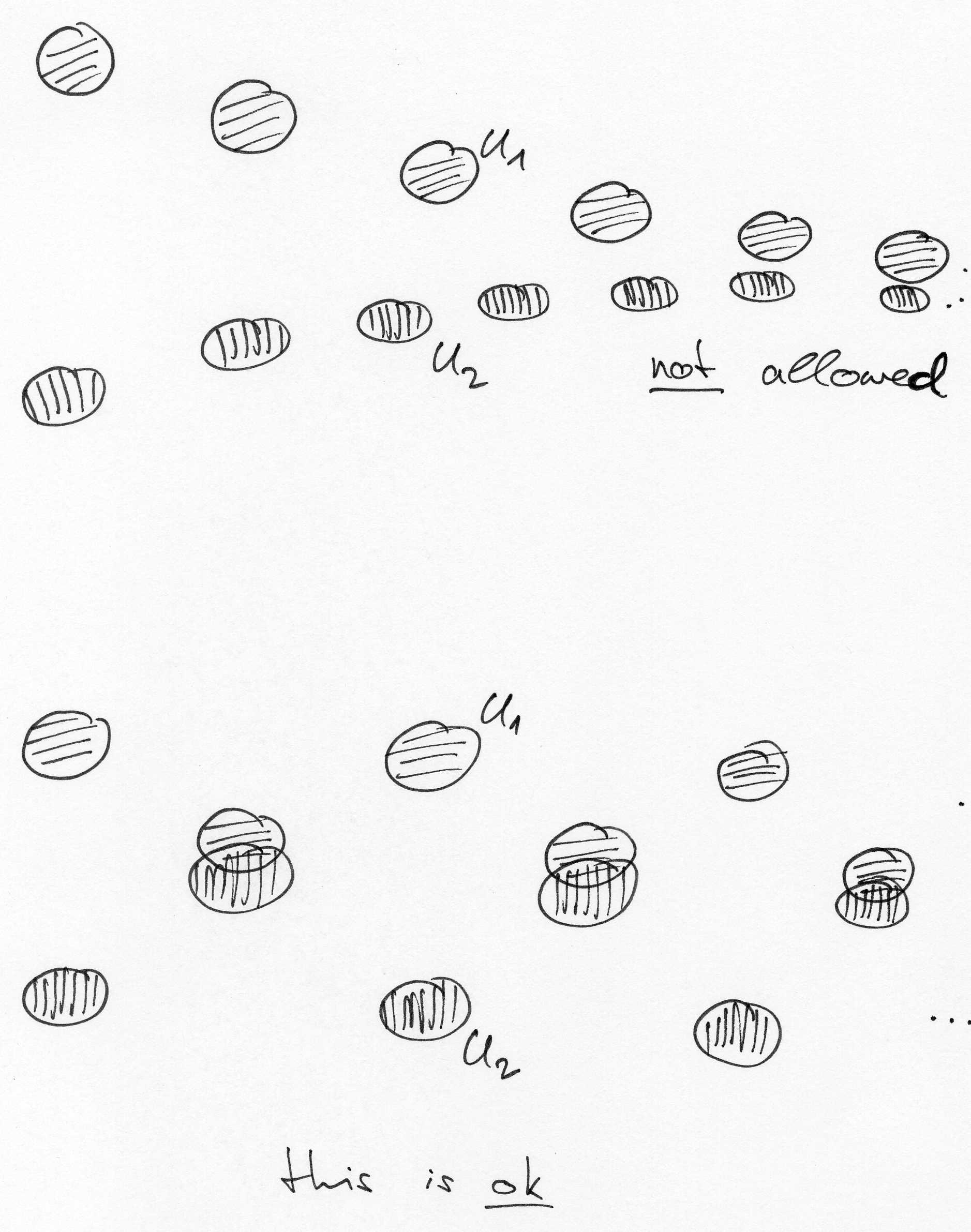}
\caption{Illustration for Lemma \ref{lem:suitable_coloring_cover_M}.\ref{cover_i}.}
\label{fig:not_allowed_cover}
\end{figure}

\begin{proof}
Let us first show how to get a cover of $M$ satisfying Point~\ref{cover_i} from the lemma.

We triangulate $M$ via the above Theorem \ref{thm:triangulation_bounded_geometry}. Then we may take the vertices of this triangulation as our collection of points $\{x_i\}_{i \in I}$ and set $\varepsilon$ to $2/3$ of the length of an edge multiplied with the constant $C$ which we get since the metric derived from barycentric coordinates is bi-Lipschitz equivalent to the metric derived from the Riemannian structure.

Two balls $B_\varepsilon(x_i)$ and $B_\varepsilon(x_j)$ for $x_i \not= x_j$ intersect if and only if $x_i$ and $x_j$ are adjacent vertices, and in the case that they are not adjacent, these balls are a uniform distance apart from each other. Hence it is possible to find a coloring of all these balls $\{B_\varepsilon(x_i)\}_{i \in I}$ with only finitely many colors having the claimed Property~\ref{cover_i}: apply Lemma \ref{lem:coloring_graph} to the covering $\{B_\varepsilon(x_i)\}_{i \in I}$ which has finite multiplicity due to bounded geometry.

To prove Point~\ref{cover_ii}, we replace in our cover of $M$ the balls $B_\varepsilon(x_i)$ with slightly differently chosen open subsets, as shown in the $2$-dimensional case in Figure~\ref{fig_fancy_coloring} (we are working in a triangulation of $M$ as above in the proof of Point~\ref{cover_i}).
\end{proof}

\begin{figure}[ht]
\centering
\includegraphics[scale=0.3]{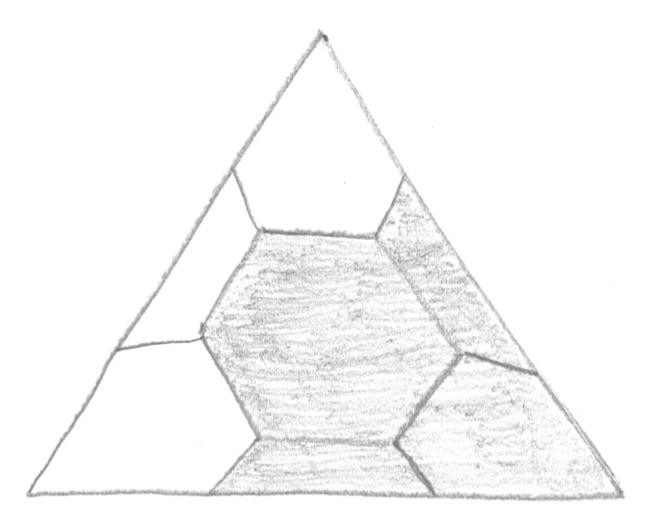}
\caption{Illustration for Lemma \ref{lem:suitable_coloring_cover_M}.\ref{cover_ii}.}
\label{fig_fancy_coloring}
\end{figure}

\subsection{Uniform pseudodifferential operators}
\label{sec_UPDOs}

In this section we will recall the definition of uniform pseudodifferential operators and some basic properties of them. This class of pseudodifferential operators was introduced by the author in his Ph.D.~thesis \cite{engel_phd}, but similar classes were also considered by Shubin \cite{shubin} and Kordyukov \cite{kordyukov}.

Let $M^m$ be an $m$-dimensional manifold of bounded geometry and let $E$ and $F$ be two vector bundles of bounded geometry over $M$.

\begin{defn}\label{defn:pseudodiff_operator}
An operator $P\colon C_c^\infty(E) \to C^\infty(F)$ is a \emph{uniform pseudodifferential operator of order $k \in \IZ$}, if with respect to a uniformly locally finite covering $\{B_{2\varepsilon}(x_i)\}$ of $M$ with normal coordinate balls and corresponding subordinate partition of unity $\{\varphi_i\}$ as in Lemma \ref{lem:nice_coverings_partitions_of_unity} we can write
\begin{equation}
\label{eq:defn_pseudodiff_operator_sum}
P = P_{-\infty} + \sum_i P_i
\end{equation}
satisfying the following conditions:
\begin{itemize}
\item $P_{-\infty}$ is a quasilocal smoothing operator,\footnote{That is to say, for all $k,l \in \IN_0$ we have that $P_{-\infty}\colon H^{-k}(E) \to H^l(F)$ has the following propety: there is a function $\mu\colon \IR_{> 0} \to \IR_{\ge 0}$ with $\mu(R) \to 0$ for $R \to \infty$ and such that for all $L \subset M$ and all $u \in H^{-k}(E)$ with $\supp u \subset L$ we have $\|A u\|_{H^l, M - B_R(L)} \le \mu(R) \cdot \|u\|_{H^{-k}}$.}
\item for all $i$ the operator $P_i$ is with respect to synchronous framings of $E$ and $F$ in the ball $B_{2\varepsilon}(x_i)$ a matrix of pseudodifferential operators on $\IR^m$ of order $k$ with support\footnote{An operator $P$ is \emph{supported in a subset $K$}, if $\supp Pu \subset K$ for all $u$ in the domain of $P$ and if $Pu = 0$ whenever we have $\supp u \cap K = \emptyset$.} in $B_{2\varepsilon}(0) \subset \IR^m$, and
\item the constants $C_i^{\alpha \beta}$ appearing in the bounds
\[\|D_x^\alpha D_\xi^\beta p_i(x,\xi)\| \le C^{\alpha \beta}_i (1 + |\xi|)^{k - |\beta|}\]
of the symbols of the operators $P_i$ can be chosen to not depend on $i$, i.e., there are $C^{\alpha \beta} < \infty$ such that
\begin{equation}
\label{eq:uniformity_defn_PDOs}
C^{\alpha \beta}_i \le C^{\alpha \beta}
\end{equation}
for all multi-indices $\alpha, \beta$ and all $i$.\qedhere
\end{itemize}
\end{defn}

To define ellipticity we have to recall the definition of symbols. We let $\pi^\ast E$ and $\pi^\ast F$ denote the pull-back bundles of $E$ and $F$ to the cotangent bundle $\pi\colon T^\ast M \to M$ of the $m$-dimensional manifold $M$.

\begin{defn}\label{defnnkdf893}
Let $p$ be a section of the bundle $\Hom(\pi^\ast E, \pi^\ast F)$ over $T^\ast M$.
\begin{itemize}
\item We call $p$ a \emph{symbol of order $k \in \IZ$}, if the following holds: choosing a uniformly locally finite covering $\{ B_{2 \varepsilon}(x_i) \}$ of $M$ through normal coordinate balls and corresponding subordinate partition of unity $\{ \varphi_i \}$ as in Lemma \ref{lem:nice_coverings_partitions_of_unity}, and choosing synchronous framings of $E$ and $F$ in these balls $B_{2\varepsilon}(x_i)$, we can write $p$ as a uniformly locally finite sum $p = \sum_i p_i$, where $p_i(x,\xi) := p(x,\xi) \varphi(x)$ for $x \in M$ and $\xi \in T^\ast_x M$, and interpret each $p_i$ as a matrix-valued function on $B_{2 \varepsilon}(x_i) \times \IC^m$. Then for all multi-indices $\alpha$ and $\beta$ there must exist a constant $C^{\alpha \beta} < \infty$ such that for all $i$ and all $x, \xi$ we have
\begin{equation}\label{eq:symbol_uniformity}
\|D^\alpha_x D^\beta_\xi p_i(x,\xi) \| \le C^{\alpha \beta}(1 + |\xi|)^{k - |\beta|}.
\end{equation}
\item We will call $p$ \emph{elliptic}, if there is an $R > 0$ such that $p|_{| \xi | > R}$\footnote{We restrict $p$ to the bundle $\Hom(\pi^\ast E, \pi^\ast F)$ over the space $\{(x,\xi) \in T^\ast M \ | \ |\xi| > R\} \subset T^\ast M$.} is invertible and this inverse $p^{-1}$ satisfies the Inequality \eqref{eq:symbol_uniformity} for $\alpha, \beta = 0$ and order $-k$ (and of course only for $|\xi| > R$ since only there the inverse is defined). Note that as in the compact case it follows that $p^{-1}$ satisfies the Inequality \eqref{eq:symbol_uniformity} for all multi-indices $\alpha$, $\beta$.\qedhere
\end{itemize}
\end{defn}

\begin{defn}\label{defn:elliptic_operator}
Let $P$ be a uniform pseudodifferential operator. We will call $P$ \emph{elliptic}, if its principal symbol $\sigma(P)$ is elliptic.
\end{defn}

The main fact about elliptic operators that we will need later is the following one. Of course ellipticity is also crucially used to show that we can define a uniform $K$-homology class for such operators (see Example~\ref{ex_K_hom_classes}).

\begin{cor}[{\cite[Corollary 2.47]{engel_indices_UPDO}}]\label{cor:schwartz_function_of_PDO_quasilocal_smoothing}
Let $P$ be a symmetric and elliptic uniform pseudodifferential operator of positive order.

If $f$ is a Schwartz function, then $f(P)$ is a quasi-local smoothing operator.
\end{cor}

\subsection{Uniform \texorpdfstring{$K$}{K}-homology and uniform \texorpdfstring{$K$}{K}-theory}
\label{sec_uniform_K}

Let us start with uniform $K$-homology. For this we first have to recall briefly the notion of multigraded Hilbert spaces. They arise as $L^2$-spaces of vector bundles on which Clifford algebras act.

\begin{itemize}
\item A \emph{graded Hilbert space} is a Hilbert space $H$ with a decomposition $H = H^+ \oplus H^-$ into closed, orthogonal subspaces. This is equivalent to the existence of a \emph{grading operator} $\epsilon$ (a selfadjoint unitary) such that its $\pm 1$-eigenspaces are $H^\pm$.
\item If $H$ is a graded space, then its \emph{opposite} is the graded space $H^\op$ with underlying vector space $H$ but with the reversed grading, i.e., $(H^\op)^+ = H^-$ and $(H^\op)^- = H^+$. This is equivalent to $\epsilon_{H^\op} = -\epsilon_H$.
\item An operator on a graded space $H$ is called \emph{even} if it maps $H^\pm$ again to $H^\pm$, and it is called \emph{odd} if it maps $H^\pm$ to $H^\mp$. Equivalently, an operator is even if it commutes with the grading operator $\epsilon$ of $H$, and it is odd if it anti-commutes with it.
\end{itemize}

\begin{defn}
Let $p \in \IN_0$.
\begin{itemize}
\item A \emph{$p$-multigraded Hilbert space} is a graded Hilbert space equipped with $p$ odd unitary operators $\epsilon_1, \ldots, \epsilon_p$ such that $\epsilon_i \epsilon_j + \epsilon_j \epsilon_i = 0$ for $i \not= j$, and $\epsilon_j^2 = -1$ for all $j$.
\item Note that a $0$-multigraded Hilbert space is just a graded Hilbert space, and by convention a $(-1)$-multigraded Hilbert space is an ungraded one.
\item Let $H$ be a $p$-multigraded Hilbert space. Then an operator on $H$ will be called \emph{multigraded}, if it commutes with the multigrading operators $\epsilon_1, \ldots, \epsilon_p$ of $H$.\qedhere
\end{itemize}
\end{defn}

To define uniform Fredholm modules we will need the following notions. Let us define
\begin{equation*}
\LLip_R(X) := \{ f \in C_c(X) \ | \ f \text{ is }L\text{-Lipschitz}, \diam(\supp f) \le R \text{ and } \|f\|_\infty \le 1\}.
\end{equation*}

\begin{defn}[{\cite[Definition 2.3]{spakula_uniform_k_homology}}]\label{defn:uniform_operators}
Let $T \in \IB(H)$ be an operator on a Hilbert space $H$ and $\rho\colon C_0(X) \to \IB(H)$ a representation.

We say that $T$ is \emph{uniformly locally compact}, if for every $R, L > 0$ the collection
\[\{\rho(f)T, T\rho(f) \ | \ f \in \LLip_R(X)\}\]
is uniformly approximable.\footnote{A collection of operators $\mathcal{A} \subset \IK(H)$ is said to be \emph{uniformly approximable}, if for every $\varepsilon > 0$ there is an $N > 0$ such that for every $T \in \mathcal{A}$ there is a rank-$N$ operator $k$ with $\|T - k\| < \varepsilon$.}

We say that $T$ is \emph{uniformly pseudolocal}, if for every $R, L > 0$ the collection
\[\{[T, \rho(f)] \ | \ f \in \LLip_R(X)\}\]
is uniformly approximable.
\end{defn}

\begin{defn}[Multigraded uniform Fredholm modules, cf.~{\cite[Definition 2.6]{spakula_uniform_k_homology}}]
Let $p \in \IZ_{\ge -1}$. A triple $(H,\rho,T)$ consisting of
\begin{itemize}
\item a separable $p$-multigraded Hilbert space $H$,
\item a representation $\rho\colon C_0(X) \to \IB(H)$ by even, multigraded operators, and
\item an odd multigraded operator $T \in \IB(H)$ such that
\begin{itemize}
\item the operators $T^2 - 1$ and $T - T^\ast$ are uniformly locally compact and
\item the operator $T$ itself is uniformly pseudolocal
\end{itemize}
\end{itemize}
is called a \emph{$p$-multigraded uniform Fredholm module over $X$}.
\end{defn}

\begin{defn}[Uniform $K$-homology, {\cite[Definition 2.13]{spakula_uniform_k_homology}}]
We define the \emph{uniform $K$-homology group $K_{p}^u(X)$} of any locally compact, separable metric space $X$ to be the abelian group generated by unitary equivalence classes of $p$-multigraded uniform Fredholm modules with the relations:
\begin{itemize}
\item if $x$ and $y$ are operator homotopic\footnote{A collection $(H, \rho, T_t)$ of uniform Fredholm modules is called an \emph{operator homotopy} if $t \mapsto T_t \in \IB(H)$ is norm continuous.}, then $[x] = [y]$, and
\item $[x] + [y] = [x \oplus y]$,
\end{itemize}
where $x$ and $y$ are $p$-multigraded uniform Fredholm modules.
\end{defn}

\begin{example}\label{ex_K_hom_classes}
\Spakula \cite[Theorem 3.1]{spakula_uniform_k_homology} showed that the usual Fredholm module arising from a generalized Dirac operator is uniform if we assume bounded geometry: if $D$ is a generalized Dirac operator acting on a Dirac bundle $S$ of bounded geometry over a manifold $M$ of bounded geometry, then the triple $(L^2(S), \rho, \chi(D))$, where $\rho$ is the representation of $C_0(M)$ on $L^2(S)$ by multiplication operators and $\chi$ is a normalizing function, is a uniform Fredholm module. It is multigraded if the Dirac bundle $S$ has an action of a Clifford algebra.

The author \cite[Theorem 3.39 and Proposition 3.40]{engel_indices_UPDO} generalized this to symmetric and elliptic uniform pseudodifferential operators over manifolds of bounded geometry, and also showed that this uniform $K$-homology class only depends on the principal symbol of the operator.
\end{example}

Let us now recall uniform $K$-theory, which was introduced by the author in his Ph.D.\ thesis \cite{engel_phd}.

\begin{defn}[Uniform $K$-theory]
Let $X$ be a metric space. The \emph{uniform $K$-theory groups of $X$} are defined as
\[K^p_u(X) := K_{-p}(C_u(X)),\]
where $C_u(X)$ is the $C^\ast$-algebra of bounded, uniformly continuous functions on $X$.
\end{defn}

On manifolds of bounded geometry we have an interpretation of uniform $K$-theory via isomorphism classes of vector bundles of bounded geometry. In order to state this properly, we first have to recall the needed notion of isomorphy.

Let $M$ be a manifold of bounded geometry and $E$ and $F$ two complex vector bundles equipped with Hermitian metrics and compatible connections.

\begin{defn}[$C^\infty$-boundedness / $C_b^\infty$-isomorphy of vector bundle homomorphisms]\label{defn:C_infty_bounded}
We will call a vector bundle homomorphism $\varphi\colon E \to F$ \emph{$C^\infty$-bounded}, if with respect to synchronous framings of $E$ and $F$ the matrix entries of $\varphi$ are bounded, as are all their derivatives, and these bounds do not depend on the chosen base points for the framings or the synchronous framings themself.

$E$ and $F$ will be called \emph{$C_b^\infty$-isomorphic}, if there is an isomorphism $\varphi\colon E \to F$ such that both $\varphi$ and $\varphi^{-1}$ are $C^\infty$-bounded.
\end{defn}

An important property of vector bundles over compact spaces is that they are always complemented, i.e., for every bundle $E$ there is a bundle $F$ such that $E \oplus F$ is isomorphic to the trivial bundle. Note that this fails in general for non-compact spaces. The following proposition shows that we have the analogous property for vector bundles of bounded geometry. We state it here since we will need the proposition later in this paper.

\begin{defn}[$C_b^\infty$-complemented vector bundles]
A vector bundle $E$ will be called \emph{$C_b^\infty$-complemented}, if there is some vector bundle $E^\perp$ such that $E \oplus E^\perp$ is $C_b^\infty$-isomorphic to a trivial bundle with the flat connection.
\end{defn}

\begin{prop}[{\cite[Proposition 4.13]{engel_indices_UPDO}}]\label{prop:every_bundle_complemented}
Let $M$ be a manifold of bounded geometry and let $E \to M$ be a vector bundle of bounded geometry.

Then $E$ is $C_b^\infty$-complemented.
\end{prop}

We can now state the interpretation of uniform $K$-theory on manifolds of bounded geometry via vector bundles.

\begin{thm}[Interpretation of $K^0_u(M)$, {\cite[Theorem 4.18]{engel_indices_UPDO}}]\label{thm:interpretation_K0u}
Let $M$ be a Riemannian manifold of bounded geometry and without boundary.

Then every element of $K^0_u(M)$ is of the form $[E] - [F]$, where both $[E]$ and $[F]$ are $C_b^\infty$-isomorphism classes of complex vector bundles of bounded geometry over $M$.

Moreover, every complex vector bundle of bounded geometry over $M$ defines naturally a class in $K^0_u(M)$.
\end{thm}

Note that the last statement in the above theorem is not trivial since it relies on the fact that every vector bundle of bounded geometry is suitably complemented.

\begin{thm}[Interpretation of $K^1_u(M)$, {\cite[Theorem 4.21]{engel_indices_UPDO}}]\label{thm:interpretation_K1u}
Let $M$ be a Riemannian manifold of bounded geometry and without boundary.

Then every elements of $K^1_u(M)$ is of the form $[E] - [F]$, where both $[E]$ and $[F]$ are $C_b^\infty$-isomorphism classes of complex vector bundles of bounded geometry over $S^1 \times M$ with the following property: there is some neighbourhood $U \subset S^1$ of $1$ such that $[E|_{U \times M}]$ and $[F|_{U \times M}]$ are $C_b^\infty$-isomorphic to a trivial vector bundle with the flat connection (the dimension of the trivial bundle is the same for both $[E|_{U \times M}]$ and $[F|_{U \times M}]$).

Moreover, every pair of complex vector bundles $E$ and $F$ of bounded geometry and with the above properties define a class $[E] - [F]$ in $K_u^1(M)$.
\end{thm}

We have a cap product\footnote{We need some assumptions on the space $X$ to construct the cap product. But because every space occuring in this paper will satisfy them, we have refrained from stating these assumptions explicitly.}
\[\cap \colon K_u^p(X) \otimes K_q^u(X) \to K_{q-p}^u(X).\]
Let us collect in the next proposition some properties of it.

\begin{prop}[{\cite[Proposition 4.28]{engel_indices_UPDO}}]\label{prop:properties_general_cap_product}
\mbox{}
\begin{itemize}
\item We have the formula
\begin{equation}\label{eq:general_cap_compatibility_module}
(P \otimes Q) \cap T = P \cap (Q \cap T)
\end{equation}
for all elements $P, Q \in K_u^\ast(X)$ and $T \in K_\ast^u(X)$, where $\otimes$ is the internal product\footnote{If the classes are represented by vector bundles, then the internal product is just given by the tensor product bundle.} on uniform $K$-theory.
\item We have the following compatibility with the external products:
\begin{equation}\label{eq:compatibility_cap_external}
(P \times Q) \cap (S \times T) = (-1)^{qs} (P \cap S) \times (Q \cap T),
\end{equation}
where $P \in K_u^p(X)$, $Q \in K_u^q(X)$ and $S \in K^u_s(X)$, $T \in K^u_t(X)$.
\item If $E \to M$ is a vector bundle of bounded geometry over a manifold $M$ of bounded geometry and $D$ an operator of Dirac type over $M$, then we have
\begin{equation}\label{eq:cap_twisted_Dirac}
[E] \cap [D] = [D_E] \in K_\ast^u(M),
\end{equation}
where $D_E$ is the twisted operator.
\end{itemize}
\end{prop}

The main reason why we have recalled the cap product is the following duality result:

\begin{thm}[Uniform $K$-\Poincare duality, {\cite[Theorem 4.29]{engel_indices_UPDO}}]
\label{thm:Poincare_duality_K}
Let $M$ be an $m$-dimensional spin$^c$ manifold of bounded geometry and without boundary.

Then the cap product $- \cap [M] \colon K_u^\ast(M) \to K^u_{m-\ast}(M)$ with its uniform $K$-fundamental class $[M] \in K_m^u(M)$ is an isomorphism.
\end{thm}

\section{Uniform homology theories and Chern characters}

In Section~\ref{sec_cyclic_cocycles} we will recall the definition of (periodic) cyclic cohomology and construct the Chern--Connes characters $\ch\colon K_\ast^u(M) \dashrightarrow \HPucont^\ast(\Winftyone(M))$. In Section~\ref{sec:u_de_Rham_currents} we will then map further into uniform de Rham homology $\HudR_\ast(M)$ and prove various additional results, e.g., that we have an isomorphism $\HPucont^{\ast}(\Winftyone(M)) \cong \HudR_{\ast}(M)$ and \Poincare duality $\HbdR^\ast(M) \cong \HudR_{m-\ast}(M)$. At the end of Section~\ref{sec:u_de_Rham_currents} we discuss the Chern character $\ch \colon K^\ast_u(M) \to \HudRco^\ast(M)$ and the whole Section~\ref{sec:chern_isos} is devoted to the proof of the Chern character isomorphism theorem.

\subsection{Cyclic cocycles of uniformly finitely summable modules}
\label{sec_cyclic_cocycles}

The goal of this section is to construct the homological Chern character maps from uniform $K$-homology $K_\ast^u(M)$ of $M$ to continuous periodic cyclic cohomology $\HPucont^\ast(\Winftyone(M))$ of the Sobolev space $\Winftyone(M)$.

First we will recall the definition of Hochschild, cyclic and periodic cyclic cohomology of a (possibly non-unital) complete locally convex algebra $A$\footnote{We consider here only algebras over the field $\IC$. Furthermore, we assume that multiplication in $A$ is jointly continuous.}. The classical reference for this is, of course, Connes' seminal paper \cite{connes_noncomm_diff_geo}. The author also found Khalkhali's book \cite{khalkhali_basic} a useful introduction to these matters.

\begin{defn}
The \emph{continuous Hochschild cohomology} $\HHucont^\ast(A)$ of $A$ is the homology of the complex
\[\Cucont^0(A) \stackrel{b}\longrightarrow \Cucont^1(A) \stackrel{b}\longrightarrow \ldots,\]
where $\Cucont^n(A) = \Hom(A^{\widehat{\otimes}(n+1)}, \IC)$ and the boundary map $b$ is given by
\begin{align*}
(b\varphi)(a_0, \ldots, a_{n+1}) = & \sum_{i=0}^n (-1)^i \varphi(a_0, \ldots, a_i a_{i+1}, \ldots, a_{n+1}) +\\
& + (-1)^{n+1} \varphi(a_{n+1} a_0, a_1, \ldots, a_n).
\end{align*}
We use the completed projective tensor product $\widehat{\otimes}$ and the linear functionals are assumed to be continuous. But we still factor out only the image of the boundary operator to define the homology, and \emph{not} the closure of the image of $b$.
\end{defn}

\begin{defn}
The \emph{continuous cyclic cohomology} $\HCucont^\ast(A)$ of $A$ is the homology of the following subcomplex of the Hochschild cochain complex:
\[\Clucont^0(A) \stackrel{b}\longrightarrow \Clucont^1(A) \stackrel{b}\longrightarrow \ldots,\]
where $\Clucont^n(A) = \{\varphi \in \Cucont^n(A)\colon \varphi(a_n, a_0, \ldots, a_{n-1}) = (-1)^n \varphi(a_0, a_1, \ldots, a_n)\}$.
\end{defn}

There is a certain \emph{periodicity operator} $S\colon \HCucont^n(A) \to \HCucont^{n+2}(A)$. For the tedious definition of this operator on the level of cyclic cochains we refer the reader to Connes' original paper \cite[Lemma 11 on p.~322]{connes_noncomm_diff_geo} or to his book \cite[Lemma 14 on p.~198]{connes_book}.

\begin{defn}
The \emph{continuous periodic cyclic cohomology} $\HPucont^\ast(A)$ of $A$ is defined as the direct limit
\[\HPucont^\ast(A) = \underrightarrow{\lim} \ \HCucont^{\ast+2n}(A)\]
with respect to the maps $S$.
\end{defn}

Let $(H, \rho, T)$ be a graded uniform Fredholm module over $M$ and denote by $\epsilon$ the grading automorphism of the graded Hilbert space $H$. Moreover, assume that $(H, \rho, T)$ is involutive\footnote{Recall that a Fredholm module $(H,\rho,T)$ is called involutive if $T=T^*$, $\|T\|\le 1$ and $T^2=1$.} and uniformly $p$-summable, where the latter means $\sup_{f \in \LLip_R(M)} \|[T, \rho(f)]\|_p < \infty$ for the Schatten $p$-norm $\|-\|_p$.

Having such an involutive, uniformly $p$-summable Fredholm module at hand we define for all $m$ with $2m+1 \ge p$ a cyclic $2m$-cocycle on $\Winftyone(M)$, i.e., on the Sobolev space of infinite order and $L^1$-integrability, by
\[\ch^{0,2m}(H, \rho, T)(f_0, \ldots, f_{2m}) := \tfrac{1}{2} (2\pi i)^m m! \trace\big( \epsilon T [T, f_0] \cdots [T, f_{2m}] \big).\]
We have the compatibility $S \circ \ch^{0,2m} = \ch^{0, 2m+2}$ and therefore we get a map
\[\ch^0\colon K^u_0(M) \dashrightarrow \HPucont^0(\Winftyone(M)).\]

The dashed arrow indicates that we do not know that every uniform, even $K$-homology class is represented by a uniformly finitely summable module, and we also do not know if the map is well-defined, i.e., if two such modules representing the same $K$-homology class will be mapped to the same cyclic cocycle class. For \spinc manifolds the first mentioned problem is solved by \Poincare duality which states that every uniform $K$-homology class may be represented by the difference of two twisted Dirac operators (which are uniformly finitely summable). But the second mentioned problem about the well-definedness is much more serious and will only be solved by the local index theorem. We will state the resolution of this problem in Corollary \ref{cor:ch_well_defined}.

Given an ungraded, involutive, uniformly $p$-summable Fredholm module $(H, \rho, T)$, we define for all $m$ with $2m \ge p$ a cyclic $(2m-1)$-cocycle on $\Winftyone(M)$ by
\begin{align*}
\ch^{1, 2m-1}(H, \rho, T & )(f_0, \ldots, f_{2m-1}) =\\
& = (2\pi i)^m \tfrac{1}{2} (2m-1) (2m-3) \cdots 3 \cdot 1 \trace\big( T [T, f_0] \cdots [T, f_{2m-1}]\big).
\end{align*}
Again, this definition is compatible with the periodicity operator $S$ and so defines a map
\[\ch^1\colon K^u_1(M) \dashrightarrow \HPucont^1(\Winftyone(M)).\]

\subsection{Uniform de Rham (co-)homology}
\label{sec:u_de_Rham_currents}

In the previous section we constructed the characters $\ch\colon K_\ast^u(M) \dashrightarrow \HPucont^\ast(\Winftyone(M))$. The first goal of this section is to map further to uniform de Rham homology $\HudR_{\ast}(M)$. In the second part of this section we will then prove \Poincare duality of the latter with bounded de Rham cohomology: $\HbdR^\ast(M) \cong \HudR_{m-\ast}(M)$. And at the end of this section we will introduce uniform de Rham cohomology and construct the uniform Chern character from uniform $K$-theory to it.

\begin{defn}\label{defn:de_rham_hom}
We define the space of \emph{uniform de Rham $p$-currents} $\Omega_p^u(M)$ to be the topological dual space of the \Frechet space $W^{\infty, 1}(\Omega^p(M))$, i.e.,
\[\Omega_p^u(M) := \Hom(W^{\infty, 1}(\Omega^p(M)), \IC).\]
Recall from Definition \ref{defn:sobolev_spaces} and Equation \eqref{eq:defn_W_infty} that $W^{\infty, 1}(\Omega^p(M))$ denotes the Sobolev space of $p$-forms whose derivatives are all $L^1$-integrable.

Since the exterior derivative $d\colon W^{\infty, 1}(\Omega^p(M)) \to W^{\infty, 1}(\Omega^{p+1}(M))$ is continuous we get a corresponding dual differential (also denoted by $d$)
\begin{equation}
\label{eqjnker4}
d\colon \Omega_p^u(M) \to \Omega_{p-1}^u(M).
\end{equation}
We define the \emph{uniform de Rham homology} $\HudR_\ast(M)$ with coefficients in $\IC$ as the homology of the complex
\[\ldots \stackrel{d}\longrightarrow \Omega_p^u(M) \stackrel{d}\longrightarrow \Omega^u_{p-1}(M) \stackrel{d}\longrightarrow \ldots \stackrel{d}\longrightarrow \Omega_0(M) \to 0,\]
where $d$ is the dual differential \eqref{eqjnker4}.
\end{defn}

\begin{defn}
We define a map $\alpha\colon \Cucont^p(\Winftyone(M)) \to \Omega_p^u(M)$ by
\[\alpha(\varphi)(f_0 d f_1 \wedge \ldots \wedge d f_p) := \frac{1}{p!} \sum_{\sigma \in \frakS_p} (-1)^\sigma \varphi(f_0, f_{\sigma(1)}, \ldots, f_{\sigma(p)}),\]
where $\frakS_p$ denotes the symmetric group on $1, \ldots, p$.
\end{defn}

The antisymmetrization that we have done in the above definition of $\alpha$ maps Hochschild cocycles to Hochschild cocycles and vanishes on Hochschild coboundaries. This means that $\alpha$ descends to a map
\[\alpha\colon \HHucont^\ast(\Winftyone(M)) \to \Omega_\ast^u(M)\]
on Hochschild cohomology.

Before we can prove that $\alpha$ is an isomorphism we need a technical lemma:

\begin{lem}\label{lem:tensor_prod_sobolev}
Let $M$ and $N$ be manifolds of bounded geometry and without boundary. Then we have
\[\Winftyone(M) \hatotimes \Winftyone(N) \cong \Winftyone(M \times N),\]
where $\hatotimes$ denotes the projective tensor product.
\end{lem}

\begin{proof}
This proof is an elaboration of P.~Michor's answer \cite{MO_michor} on MathOverflow. The reference that he gives is \cite{michor_book}: combining the Theorem on p.~78 in it with Point (c) on top of the same page we get the isomorphism $L^1(M) \hatotimes L^1(N) \cong L^1(M \times N)$. This result was first proven by Chevet \cite{chevet}.

Now let us generalize this to incorporate derivatives. In \cite[End of Section 6]{kriegl_michor_rainer} it is proven\footnote{To be concrete, they proved it only for Euclidean space, but the argument is the same for manifolds of bounded geometry.} that we have a continuous inclusion $\Winftyone(M) \hatotimes \Winftyone(N) \to \Winftyone(M \times N)$. Note that we have to use \cite[Th\'{e}or\`{e}me 1 on p.~124]{chevet} to conclude that the family of seminorms used in \cite{kriegl_michor_rainer} for $\Winftyone(M) \hatotimes \Winftyone(N)$ generates indeed the projective tensor product topology.

It remains to show that $\Winftyone(M \times N) \to \Winftyone(M) \hatotimes \Winftyone(N)$ is continuous. For this we will use the fact that we may represent the projective tensor product norm on the algebraic tensor product $E \otimes_{\mathrm{alg}} F$ of two Banach spaces by
\begin{equation*}
\|u\|_{E \hatotimes F} = \inf \Big\{ \sum \|x_i\|_E \|y_i\|_F \Big\},
\end{equation*}
where the infimum ranges over all representations $u=\sum_i x_i \otimes y_i$. In our case now note that we have for $w := \sum_i (\nabla_X p_i) \otimes q_i$, where $X$ is a vector field on $M$ with $\|X\|_\infty \le 1$, the chain of inequalities
\begin{align}
\|w\|_{L^1(M) \hatotimes L^1(N)} & = \left\| \sum (\nabla_X p_i) \otimes q_i \right\|_{L^1(M) \hatotimes L^1(N)}\notag\\
& \le C \left\| \sum (\nabla_X p_i) \cdot q_i \right\|_{L^1(M\times N)}\notag\\
& \le C \|\sum p_i \cdot q_i\|_{W^{1,1}(M\times N)},\label{eq:proj_tensor_prod}
\end{align}
where the first inequality comes from the fact $L^1(M) \hatotimes L^1(N) \cong L^1(M \times N)$ which we already know. Now for $v := \sum_i s_i \otimes t_i$ we have
\begin{align}
\|v\|_{W^{1,1}(M) \hatotimes L^1(N)} & = \Big\| \sum s_i \otimes t_i \Big\|_{W^{1,1}(M) \hatotimes L^1(N)}\label{eq:proj_tensor_2}\\
& = \inf \Big\{ \sum \big( \|x_i\|_{L^1(M)} + \|\nabla x_i\|_{L^1(M)} \big) \|y_i\|_{L^1(N)} \Big\}\notag\\
& = \underbrace{\inf \Big\{ \sum \|x_i\|_{L^1(M)} \|y_i\|_{L^1(N)} \Big\}}_{= \|v\|_{L^1(M) \hatotimes L^1(N)} \le C\|v\|_{L^1(M \times N)}} + \inf \Big\{ \sum \|\nabla x_i\|_{L^1(M)} \|y_i\|_{L^1(N)} \Big\},\notag
\end{align}
where the infima run over all representations $\sum_i x_i \otimes y_i$ of $v$. Furthermore, for a fixed compactly supported vector field $X$ with $\|X\|_\infty \le 1$ we have
\begin{equation}
\label{eq:infima_equal}
\inf_{\mathcal{A}} \Big\{ \sum \|\nabla_X x_i\|_{L^1(M)} \|y_i\|_{L^1(N)} \Big\} = \inf_{\mathcal{B}} \Big\{ \sum \|e_i\|_{L^1(M)} \|f_i\|_{L^1(N)} \Big\},
\end{equation}
where $\mathcal{A}$ is the set of all representations $\sum_i x_i \otimes y_i$ of $v = \sum_i s_i \otimes t_i$ and $\mathcal{B}$ the set of all representations $\sum_i e_i \otimes f_i$ of $\sum_i (\nabla_X s_i) \otimes t_i$. This equality holds because every element of $\mathcal{A}$ gives rise to an element of $\mathcal{B}$ by deriving the first component and also vice versa by integrating it. By Inequality \eqref{eq:proj_tensor_prod} we now get that the infima in Equation \eqref{eq:infima_equal} are less than or equal to $C\|v\|_{W^{1,1}(M \times N)}$. Since this holds for any vector field $X$ with $\|X\|_\infty \le 1$ we can combine it now with Estimate \eqref{eq:proj_tensor_2} to get
\[\|v\|_{W^{1,1}(M) \hatotimes L^1(N)} \le 2C\|v\|_{W^{1,1}(M \times N)}.\]

We iterate the argument to get estimates for all higher derivatives and also for the second component. This proves the claim that the map $\Winftyone(M \times N) \to \Winftyone(M) \hatotimes \Winftyone(N)$ is continuous and therefore completes the whole proof.
\end{proof}

\begin{thm}
For any Riemannian manifold $M$ of bounded geometry and without boundary the map $\alpha\colon \HHucont^p(\Winftyone(M)) \to \Omega_p^u(M)$ is an isomorphism for all $p$.
\end{thm}

\begin{proof}
The proof is analogous to the one given in \cite[Lemma 45a on page 128]{connes_noncomm_diff_geo} for the case of compact manifolds. We describe here only the places where we have to adjust it for non-compact manifolds.

The proof in \cite{connes_noncomm_diff_geo} relies heavily on Lemma 44 there. First note that direct sums, tensor products and duals of vector bundles of bounded geometry are again of bounded geometry. Since the tangent and cotangent bundle of a manifold of bounded geometry have, of course, bounded geometry, the bundles $E_k$ occuring in Lemma 44 of \cite{connes_noncomm_diff_geo} have bounded geometry.

Furthermore, \cite[Lemma 44]{connes_noncomm_diff_geo} needs a nowhere vanishing vector field on $M$, and since we are working here in the bounded geometry setting we need for our proof a nowhere vanishing vector field of norm one at every point and with bounded derivatives. Since we can without loss of generality assume that our manifold is non-compact (otherwise we are in the usual setting where the result that we want to prove is already known), we can always contruct a nowhere vanishing vector field on $M$: we just pick a generic vector field with isolated zeros and then move the vanishing points to infinity. But if we normalize this vector field to norm one at every point, then it will usually have unbounded derivatives (since we moved the vanishing points infinitely far, i.e., we disturbed the derivatives arbitrarily large). Fortunately, Weinberger proved in \cite[Theorem 1]{weinberger_euler} that on a manifold $M$ of bounded geometry a nowhere vanishing vector field of norm one and with bounded derivatives exists if and only if the Euler class $e(M) \in \HbdR^m(M)$ vanishes (the latter group denotes the top-dimensional bounded de Rham cohomology of $M$; see Definition \ref{defn:bounded_de_rham_coho}). So if the Euler class of $M$ vanishes, we are ok and can move on with our proof. If the Euler class does not vanish, then we have to use the same trick that already Connes used to prove Lemma 45a in \cite{{connes_noncomm_diff_geo}}: we take the product with $S^1$.

Moreover, we need the isomorphism $\Winftyone(M) \hatotimes \Winftyone(M) \cong \Winftyone(M \times M)$. This is exactly the content of the above Lemma \ref{lem:tensor_prod_sobolev}.

The fact that the modules $\mathcal{M}_k =\Winftyone(M \times M, E_k)$ are topologically projective, i.e., are direct summands of topological modules of the form $\mathcal{M}^\prime_k = \Winftyone(M \times M) \hatotimes \mathcal{E}_k$, where $\mathcal{E}_k$ are complete locally convex vector spaces, follows from the fact that every vector bundle $F$ of bounded geometry is $C_b^\infty$-complemented, i.e., there is a vector bundle $G$ of bounded geometry such that $F \oplus G$ is $C_b^\infty$-isomorphic to a trivial bundle with the flat connection. This is stated in Proposition \ref{prop:every_bundle_complemented}.

With the above notes in mind, the proof of \cite[Lemma 45a on page 128]{connes_noncomm_diff_geo} for the case of compact manifolds works also for non-compact manifolds in our setting here. If there are constructions to be done in the proof we have to do them uniformly (e.g., controlling derivatives uniformly in the points of the manifold) by using the bounded geometry of $M$.
\end{proof}

The inverse map $\beta\colon \Omega^u_p(M) \to \HHucont^p(\Winftyone(M))$ of $\alpha$ is given by
\[\beta(C)(f_0, f_1, \ldots, f_p) = C(f_0 d f_1 \wedge \ldots \wedge d f_p).\]

Now the proofs of Lemma 45b and Theorem 46 in \cite{connes_noncomm_diff_geo} translate without change to our setting here so that we finally get:

\begin{thm}\label{thm:computation_cyclic_cohomology}
Let $M$ be a Riemannian manifold with bounded geometry and no boundary.

For each $n \in \IN_0$ the continuous cyclic cohomology $\HCucont^n(\Winftyone(M))$ is canonically isomorphic to
\[Z_n^u(M) \oplus \HudR_{n-2}(M) \oplus \HudR_{n-4}(M) \oplus \ldots,\]
where $Z_n^u(M) \subset \Omega_n^u(M)$ is the subspace of closed currents.

The periodicity operator $S\colon \HCucont^n(\Winftyone(M)) \to \HCucont^{n+2}(\Winftyone(M))$ is given under the above isomorphism as the map that sends cycles of $Z_n^u(M)$ to their homology classes.

And last, since periodic cyclic cohomology is the direct limit of cyclic cohomology, we finally get
\[\alpha_\ast\colon \HPucont^{\mathrm{ev/odd}}(\Winftyone(M)) \stackrel{\cong}\longrightarrow \HudR_{\mathrm{ev/odd}}(M).\]
We denote this isomorphism by $\alpha_\ast$ since it is induced from the map $\alpha$ defined above.
\end{thm}

Let us now get to the dual cohomology theory to uniform de Rham homology.

\begin{defn}[Bounded de Rham cohomology]
\label{defn:bounded_de_rham_coho}
Let $\Omega_b^p(M)$ denote the vector space of $p$-forms on $M$, which are bounded in the norm
\[\|\gamma\| := \sup_{x \in M} \{ \| \gamma(x) \| + \| d \gamma(x) \| \}.\]
The \emph{bounded de Rham cohomology} $\HbdR^\ast(M)$ is defined as the homology of the corresponding complex.
\end{defn}

For an oriented manifold the \Poincare duality map between bounded de Rham cohomology and uniform de Rham homology is defined as the map induced by the following map on forms:
\begin{equation}
\label{eq_PD_coho}
\Omega^p_b(M) \to \Omega^u_{m-p}(M), \ \gamma \mapsto \big( \omega \mapsto \int_M \omega \wedge \gamma \big).
\end{equation}

\begin{thm}\label{thm:Poincare_de_Rham}
Let $M^m$ be an oriented Riemannian manifold of bounded geometry and without boundary.

Then the \Poincare duality map \eqref{eq_PD_coho} induces an isomorphism
\[\HbdR^\ast(M) \stackrel{\cong}\longrightarrow \HudR_{m-\ast}(M)\]
between bounded de Rham cohomology of $M$ and uniform de Rham homology of $M$.
\end{thm}

\begin{proof}
We will do a Mayer--Vietoris induction, similar as in the proof of \Poincare duality between uniform $K$-theory and uniform $K$-homology in \cite[Section 4.4]{engel_indices_UPDO}.

We invoke Lemma~\ref{lem:suitable_coloring_cover_M} to get a cover of $M$ by open subsets having Properties~\ref{cover_i} and~\ref{cover_ii} from that lemma.\footnote{With the additional property that the boundaries of the open subsets are smooth, but it is clear that we can arrange this.} We use the notation $U_j$ and $U_K$ from it, and the induction will be over the index $j$ (and hence the proof will only consist of finitely many induction steps).

We have to show that we have the Mayer--Vietoris sequences. The arguments are the same as in the case of compact manifolds, and we will only mention where we have to be cautios because we are working in the setting of uniform theories. We will only discuss the case of bounded de Rham cohomology, since the additional arguments (because of the uniform situation) in the case of uniform de Rham homology are similar.

For bounded de Rham cohomology we have to show that the following sequence is exact in order to get a Mayer--Vietoris sequence:
\begin{equation}
\label{eq_MV_de_Rham}
0 \to \Omega_{b}^*(U_K \cup U_{k+1}) \to \Omega_{b}^*(U_K) \oplus \Omega_{b}^*(U_{k+1}) \to \Omega_{b}^*(U_K \cap U_{k+1}) \to 0.
\end{equation}
The crucial step is to show that the map $\Omega_{b}^*(U_K) \oplus \Omega_{b}^*(U_{k+1}) \to \Omega_{b}^*(U_K \cap U_{k+1})$ is surjectice. The usual argument in the case of compact manifolds uses a partition of unity, and here we have to make sure now that the partition of unity has uniformly bounded derivatives of all orders. The reason that we can construct such a partition of unity here is because of Property~\ref{cover_ii} of Lemma~\ref{lem:suitable_coloring_cover_M}.

That the above defined \Poincare duality map \eqref{eq_PD_coho} is a natural transformation from one Mayer--Vietoris sequence to the other may be proved analogously as in the case of compact manifolds; see, e.g., \cite[Exercise 16-6]{lee_smooth}.

And finally, let us discuss the first step of the induction. We have collections $U_1$, $U_2$ and $U_1 \cap U_2$ which are each a uniformly disjoint union of open subsets of $M$ which have a uniform bound on their diameters. So all three sets are boundedly homotopy equivalent\footnote{\label{footnote:boundedly_homotopic}Let $f, g\colon M \to N$ be two maps of bounded dilatation. We say that they are \emph{boundedly homotopic}, if there is a homotopy $H\colon M \times [0,1] \to N$ from $f$ to $g$, which itself is of bounded dilatation. Recall that a map $h$ has \emph{bounded dilatation}, if $\|h_\ast V\| \le C \|V\|$ for all tangent vectors $V$. Bounded homotopy invariance of bounded de Rham cohomology was shown by the author in \cite[Corollary 5.26]{engel_phd}.} to an infinite collection of open balls, for which we already know from the case of compact manifolds that the \Poincare duality map is an isomorphism.
\end{proof}

Bounded de Rham cohomology does not perfectly fit the setting in this paper since the condition that the exterior derivative of a form is bounded does not imply that in local coordinates the coefficient functions have a uniformly bounded first derivative, and it also does not say anything about the higher derivatives. Hence the following definition and proposition.

\begin{defn}
The \emph{uniform de Rham cohomology $\HudRco^\ast(M)$} of a Riemannian manifold $M$ of bounded geometry is defined by using the complex of uniform $C^\infty$-spaces\footnote{see Definition \ref{defn:uniform_frechet_spaces}} $C^\infty_b(\Omega^\ast(M))$, i.e., differential forms on $M$ which have in normal coordinates bounded coefficient functions and all derivatives of them are also bounded.
\end{defn}

\begin{prop}\label{prop:uniform_bounded_dRco}
Let $M$ be a manifold of bounded geometry and without boundary. Then we have
\[\HudRco^\ast(M) \cong \HbdR^\ast(M).\]
\end{prop}

\begin{proof}
The proof is analogous to the one of Theorem~\ref{thm:Poincare_de_Rham} --- the important thing is that we have Mayer--Vietoris sequences and the argument given in the proof of Theorem~\ref{thm:Poincare_de_Rham} for bounded de Rham cohomology also applies to uniform de Rham cohomology.
\end{proof}

\begin{thm}[Existence of the uniform Chern character]\label{thm_Chern_character}
Let $M$ be a Riemannian manifold of bounded geometry and without boundary.

Then we have a ring homomorphism $\ch \colon K^\ast_u(M) \to \HudRco^\ast(M)$ with
\[\ch(K^0_u(M)) \subset \HudRco^\ev(M) \text{ and } \ch(K^1_u(M)) \subset \HudRco^\odd(M).\]
\end{thm}

\begin{proof}
The Chern character is defined via Chern--Weil theory. That we get uniform forms if we use vector bundles of bounded geometry is proven in \cite[Theorem 3.8]{roe_index_1} and so we get a map $\ch \colon K^0_u(M) \to \HudRco^\ev(M)$. That we also have a map $\ch \colon K^1_u(M) \to \HudRco^\odd(M)$ uses the description of $K^1_u(M)$ from Theorem \ref{thm:interpretation_K1u}, i.e., that it consists of suitable vector bundles over $S^1 \times M$, and a corresponding suspension isomorphism for the uniform de Rham cohomology. Details (for bounded cohomology, but for uniform cohomology it is analogous) may be found in the author's Ph.D.\ thesis \cite[Sections 5.4 \& 5.5]{engel_phd}.
\end{proof}

\subsection{Uniform Chern character isomorphism theorems}\label{sec:chern_isos}

The results of the last two sections tell us that we have constructed Chern characters $K^\ast_u(M) \to \HudRco^\ast(M)$ and $K_\ast^u(M) \to \HudR_\ast(M)$. Here we already use the Corollary~\ref{cor:ch_well_defined} further below which states that the uniform homological Chern character is well-defined. In the compact case the Chern characters are isomorphisms modulo torsion and it is natural to ask the same question here in the uniform setting. It is the goal of this section to answer this question positively.

The proofs use the same Mayer--Vietoris induction as the proof of \Poincare duality in \cite[Section~4.4]{engel_indices_UPDO} and Theorem~\ref{thm:Poincare_de_Rham}. Therefore we will discuss in this section only the parts of the proofs which need additional arguments.

The most crucial detail to discuss here is the statement of the theorem itself since we cannot just take the tensor product of the $K$-groups with the complex numbers to get isomorphisms. In turns out that we additionally have to form a certain completion of the algebraic tensor product of the $K$-groups with $\IC$. We will discuss this completion directly after the statement of the theorem.

\begin{thm}\label{thm23w2}
Let $M$ be a manifold of bounded geometry and without boundary. Then the Chern characters induce linear, continuous isomorphisms\footnote{The inverse maps are in general not continuous, because $\HudRco^\ast(M)$, respectively $\HudR_\ast(M)$, are in general (e.g., if $M$ is not compact) not Hausdorff, whereas $K_u^\ast(M) \barotimes \IC$, respectively $K^u_\ast(M) \barotimes \IC$, are. The topology on the latter spaces is defined by equipping the $K$-groups with the discrete topology and then forming the completed tensor product with $\IC$ which will be discussed after the statement of the theorem.}
\[K_u^\ast(M) \barotimes \IC \cong \HudRco^\ast(M) \text{ and } K^u_\ast(M) \barotimes \IC \cong \HudR_\ast(M).\]
\end{thm}

Let us discuss why we have to take a completion at all. Consider the beginning of the Mayer--Vietoris induction where we have to show that the Chern characters induce isomorphisms on a countably infinite collection of uniformly discretely distributed points. Let these points be indexed by a set $Y$. Then the $K$-groups of $Y$ are given by $\ell^\infty_\IZ(Y)$, the group of all bounded, integer-valued sequences indexed by $Y$, and the de Rham groups are given by $\ell^\infty(Y)$, the group of all bounded, complex valued sequences on $Y$. But since $Y$ is countably infinite we have $\ell^\infty_\IZ(Y) \otimes \IC \not\cong \ell^\infty(Y)$. Instead we have $\overline{\ell^\infty_\IZ(Y) \otimes \IC} \cong \ell^\infty(Y)$.

To define the \emph{completed topological tensor product of an abelian group with $\IC$} we will need the notion of the \emph{free (abelian) topological group}: if $X$ is any completely regular\footnote{That is to say, every closed set $K$ can be separated with a continuous function from every point $x \notin K$. Note that this does not necessarily imply that $X$ is Hausdorff.} topological space, then the free topological group $F(X)$ on $X$ is a topological group such that we have
\begin{itemize}
\item a topological embedding $X \hookrightarrow F(X)$ of $X$ as a closed subset, so that $X$ generates $F(X)$ algebraically as a free group (i.e., the algebraic group underlying the free topological group on $X$ is the free group on $X$), and we have
\item the following universal property: for every continuous map $\phi\colon X \to G$, where $G$ is any topological group, we have a unique extension $\Phi\colon F(X) \to G$ of $\phi$ to a continuous group homomorphism on $F(X)$:
\[\xymatrix{X \ar@{^{(}->}[r] \ar[d]_{\phi} & F(X) \ar@{-->}[dl]^{\exists ! \Phi}\\ G}\]
\end{itemize}
The free abelian topological group $A(X)$ has the corresponding analogous properties. Furthermore, the commutator subgroup $[F(X), F(X)]$ of $F(X)$ is closed and the quotient $F(X) / [F(X), F(X)]$ is both algebraically and topologically $A(X)$.

As an easy example consider $X$ equipped with the discrete topology. Then $F(X)$ and $A(X)$ also have the discrete topology.

It seems that free (abelian) topological groups were apparently introduced by Markov in \cite{markoff_original}. But unfortunately, the author could not obtain any (neither russian nor english) copy of this article. A complete proof of the existence of such groups was given by Markov in \cite{markoff}. Since his proof was long and complicated, several other authors gave other proofs, e.g., Nakayama in \cite{nakayama}, Kakutani in \cite{kakutani} and Graev in \cite{graev}.

Now let us construct for any abelian topological group $G$ the complete topological vector space $G \barotimes \IC$. We form the topological tensor product $G \otimes \IC$ of abelian topological groups in the usual way: we start with the free abelian topological group $A(G \times \IC)$ over the topological space $G \times \IC$ equipped with the product topology\footnote{Note that every topological group is automatically completely regular and therefore the product $G \times \IC$ is also completely regular.} and then take the quotient $A(G \times \IC) / \mathcal{N}$ of it,\footnote{Since $A(X)$ is both algebraically and topologically the quotient of $F(X)$ by its commutator subgroup, we could also have started with $F(G \times \IC)$ and additionally put the commutator relations into $\mathcal{N}$.} where $\mathcal{N}$ is the closure of the normal subgroup generated by the usual relations for the tensor product.\footnote{That is to say, $\mathcal{N}$ contains $(g_1 + g_2) \times r - g_1 \times r - g_2 \times r$, $g \times (r_1 + r_2) - g \times r_1 - g \times r_2$ and $zg \times r - z(g \times r)$, $g \times zr - z(g \times r)$, where $g, g_1, g_2 \in G$, $r, r_1, r_2 \in \IC$ and $z \in \IZ$.} Now we may put on $G \otimes \IC$ the structure of a topological vector space by defining the scalar multiplication to be $\lambda (g \otimes r) := g \otimes \lambda r$.

What we now got is a topological vector space $G \otimes \IC$ together with a continuous map $G \times \IC \to G \otimes \IC$ with the following universal property: for every continuous map $\phi\colon G \times \IC \to V$ into any topological vector space $V$ and such that $\phi$ is bilinear\footnote{That is to say, $\phi(\largecdot, r)$ is a group homomorphism for all $r \in \IC$ and $\phi(g, \largecdot)$ is a linear map for all $g \in G$. Note that we then also have $\phi(zg,r) = z\phi(g,r) = \phi(g,zr)$ for all $z \in \IZ$, $g \in G$ and $r \in \IC$.\label{footnote:univ_prop_TVS}}, there exists a unique, continuous linear map $\Phi\colon G \otimes \IC \to V$ such that the following diagram commutes:
\[\xymatrix{G \times \IC \ar[r] \ar[d]_{\phi} & G \otimes \IC \ar@{-->}[dl]^{\exists ! \Phi} \\ V}\]

Since every topological vector space may be completed we do this with $G \otimes \IC$ to finally arrive at $G \barotimes \IC$. Since every continuous linear map of topological vector spaces is automatically uniformly continuous, i.e., may be extended to the completion of the topological vector space, $G \barotimes \IC$ enjoys the following universal property which we will raise to a definition:

\begin{defn}[Completed topological tensor product with $\IC$]
Let $G$ be an abelian topological group. Then $G \barotimes \IC$ is a complete topological vector space over $\IC$ together with a continuous map $G \times \IC \to G \barotimes \IC$ that enjoy the following universal property: for every continuous map $\phi\colon G \times \IC \to V$ into any complete topological vector space $V$ and such that $\phi$ is bilinear\footnote{see Footnote \ref{footnote:univ_prop_TVS}}, there exists a unique, continuous linear map $\Phi\colon G \barotimes \IC \to V$ such that
\[\xymatrix{G \times \IC \ar[r] \ar[d]_{\phi} & G \barotimes \IC \ar@{-->}[dl]^{\exists ! \Phi} \\ V}\]
is a commutative diagram.
\end{defn}

We will give now two examples for the computation of $G \barotimes \IC$. The first one is easy and just a warm-up for the second which we already mentioned. Both examples are proved by checking the universal property.

\begin{examples}\label{ex:completed_tensor_prod}
The first one is $\IZ \barotimes \IC \cong \IC$.

For the second example consider the group $\ell^\infty_\IZ$ consisting of bounded, integer-valued sequences. Then $\ell^\infty_\IZ \barotimes \IC \cong \ell^\infty$.
\end{examples}

Since we want to use the completed topological tensor product with $\IC$ in a Mayer--Vietoris argument, we have to show that it transforms exact sequences to exact sequences.

So we have to show that the functor $G \mapsto G \barotimes \IC$ is exact. But we have to be careful here: though taking the tensor product with $\IC$ is exact, passing to completions is usually not---at least if the exact sequence we started with was only algebraically exact. Let us explain this a bit more thoroughly: if we have a sequence of topological vector spaces
\[\ldots \longrightarrow V_i \stackrel{\varphi_i}\longrightarrow V_{i+1} \stackrel{\varphi_{i+1}}\longrightarrow V_{i+2} \longrightarrow \ldots\]
which is exact in the algebraic sense (i.e., $\image \varphi_i = \kernel \varphi_{i+1}$), and if the maps $\varphi_i$ are continuous such that they extend to maps on the completions $\overline{V_i}$, we do not necessarily get that
\[\ldots \longrightarrow \overline{V_i} \stackrel{\overline{\varphi_i}}\longrightarrow \overline{V_{i+1}} \stackrel{\overline{\varphi_{i+1}}}\longrightarrow \overline{V_{i+2}} \longrightarrow \ldots\]
is again algebraically exact. The problem is that though we always have $\overline{\kernel \varphi_i} = \kernel \overline{\varphi_i}$, we generally only get $\overline{\image \varphi_i} \supset \image \overline{\varphi_i}$. To correct this problem we have to start with an exact sequence which is also topologically exact, i.e., we need that not only $\image \varphi_i = \kernel \varphi_{i+1}$, but we also need that $\varphi_i$ induces a topological isomorphism $V_i / \kernel \varphi_i \cong \image \varphi_i$.

To prove that in this case we get $\overline{\image \varphi_i} = \image \overline{\varphi_i}$ we consider the inverse map
\[\psi_i := \varphi_i^{-1}\colon \image \varphi_i \to V_i / \kernel \varphi_i.\]
Since $\psi_i$ is continuous (this is the point which breaks down without the additional assumption that $\varphi_i$ induces a topological isomorphism $V_i / \kernel \varphi_i \cong \image \varphi_i$), we may extend it to a map
\[\overline{\psi_i} \colon \overline{\image \varphi_i} \to \overline{V_i / \kernel \varphi_i} = \overline{V_i} / \overline{\kernel \varphi_i},\]
which obviously is the inverse to $\overline{\varphi_i} \colon \overline{V_i} / \overline{\kernel \varphi_i} \to \overline{\image \varphi_i}$ showing the desired equality $\overline{\image \varphi_i} = \image \overline{\varphi_i}$.

Coming back to our functor $G \mapsto G \barotimes \IC$, we may now prove the following lemma:

\begin{lem}\label{lem:functor_exact}
Let
\[\ldots \longrightarrow G_i \stackrel{\varphi_i}\longrightarrow G_{i+1} \stackrel{\varphi_{i+1}}\longrightarrow G_{i+2} \longrightarrow \ldots\]
be an exact sequence of topological groups and continuous maps, which is in addition topologically exact, i.e., for all $i \in \IZ$ the from $\varphi_i$ induced map $G_i / \kernel \varphi_i \to \image \varphi_i$ is an isomorphism of topological groups.

Then
\[\ldots \longrightarrow G_i \barotimes \IC \longrightarrow G_{i+1} \barotimes \IC \longrightarrow G_{i+2} \barotimes \IC \longrightarrow \ldots\]
with the induced maps is an exact sequence of complete topological vector spaces, which is also topologically exact.
\end{lem}

\begin{proof}
We first tensor with $\IC$ (without the completion afterwards). This is known to be an exact functor and our sequence also stays topologically exact. To see this last claim, we need the following fact about tensor products: if $\varphi\colon M \to M^\prime$ and $\psi\colon N \to N^\prime$ are surjective, then the kernel of $\varphi \otimes \psi \colon M \otimes M^\prime \to N \otimes N^\prime$ is the submodule given by
\[\kernel (\varphi \otimes \psi) = (\iota_M \otimes 1) \big( (\kernel \varphi) \otimes N \big) + (1 \otimes \iota_N) \big( M \otimes (\kernel \psi) \big),\]
where $\iota_M \colon \kernel \varphi \to M$ and $\iota_N \colon \kernel \psi \to N$ are the inclusion maps. We will suppress the inclusion maps from now on to shorten the notation.

We apply this with the map $\varphi \colon M \to M^\prime$ being the quotient map $G_i \to G_i / \kernel \varphi_i$ and $\psi \colon N \to N^\prime$ being the identity $\id \colon \IC \to \IC$ to get
\[\kernel (\varphi_i \otimes \id) = (\kernel \varphi_i) \otimes \IC.\]
Since we have $(\image \varphi_i) \otimes \IC = \image (\varphi_i \otimes \id)$, we get that $\varphi \otimes \id \colon G_i \otimes \IC \to G_i \otimes \IC$ induces an algebraic isomorphism $(G_i / \kernel \varphi_i) \otimes \IC \to \image \varphi_i \otimes \IC$. But this has now an inverse map given by tensoring the inverse of $G_i / \kernel \varphi_i \to \image \varphi_i$ with $\id \colon \IC \to \IC$. So the isomorphism $(G_i / \kernel \varphi_i) \otimes \IC \cong \image \varphi_i \otimes \IC$ is also topological.

Now we apply the discussion before the lemma to show that the completion of this new sequence is still exact and also topologically exact.
\end{proof}

To show $K_u^\ast(M) \barotimes \IC \cong \HudRco^\ast(M)$ it remains to construct Mayer--Vietoris sequences 
with continuous maps in them (we need this since in constructing the completed tensor product with $\IC$ we have to pass to the completion and without continuity of the maps in both the Mayer--Vietoris sequences for uniform $K$-theory and for uniform de Rham cohomology we would not be able to conclude that the squares are still commutative). If we recall from the proof of Proposition~\ref{prop:uniform_bounded_dRco} how we get the boundary maps in the Mayer--Vietoris sequence for uniform de Rham cohomology, we see that we must construct a continuous split to the last non-trivial map in the sequence \eqref{eq_MV_de_Rham}.\footnote{The referenced sequence is for bounded de Rham cohomology. In this proof here we, of course, have to use the analogous sequence for uniform de Rham cohomology.} But we proved surjectivity of this map in the usual way by using partitions of unity (with uniformly bounded derivatives). Hence we have already constructed the continuous split.

In the proof of \Poincare duality between uniform $K$-theory and uniform $K$-homology in \cite[Section~4.4]{engel_indices_UPDO} we used groups denoted by $K^\ast_u(O \subset M)$ for the Mayer--Vietoris sequence for uniform $K$-theory. These groups are defined as $K_{-\ast}(C_u(O,d))$, where $(O,d)$ is the metric space $O$ equipped with the subspace metric derived from the metric space~$M$. For the construction of the Chern character $K^\ast_u(O \subset M) \to \HudRco^\ast(O)$ we have to pass to a smooth subalgebra of $C_u(O,d)$. This will be of course $C_b^\infty(O) \subset C_u(O,d)$, which is a local $C^\ast$-algebra.\footnote{That is to say, its operator $K$-theory coincides with the operator $K$-theory of its $C^\ast$-algebra completion $C_u(O,d)$.} We have to argue now why it is a dense subalgebra: so let $f \in C_u(O,d)$ be given. Then we know from \cite[Lemma~4.36]{engel_indices_UPDO} that there is a bounded, uniformly continuous extension $F$ of $f$ to $M$. Now we use \cite[Lemma~4.7]{engel_indices_UPDO} to approximate $F$ by functions from $C_b^\infty(M)$, which will give us by restriction to $O$ an approximation of $f$ by functions from $C_b^\infty(O)$. Therefore we get an interpretation of $K^\ast_u(O \subset M)$ by vector bundles of bounded geometry over $O$ (cf.~Section~\ref{sec_uniform_K}) and may define by Chern--Weil theory (as in Theorem~\ref{thm_Chern_character}) the Chern character $K^\ast_u(O \subset M) \to \HudRco^\ast(O)$.

The last thing that we have to discuss is the small ambiguity in extending the maps $K^\ast_u(O \subset M) \otimes \IC \to \HudRco^\ast(O)$ to $K^\ast_u(O \subset M) \barotimes \IC$. It occurs because the target $\HudRco^\ast(O)$ is not necessarily Hausdorff. What we have to make sure is that the extensions we choose in the Mayer--Vietoris argument for the subsets $U_k$, respectively $U_K$, do match up, i.e., produce at the end commuting squares in the comparison of the two Mayer--Vietoris sequences via the Chern characters.

So we have finally discussed everything that we need in order to prove
\[K_u^\ast(M) \barotimes \IC \cong \HudRco^\ast(M).\]

Proving the homological version $K^u_\ast(M) \barotimes \IC \cong \HudR_\ast(M)$ is also such a Mayer--Vietoris argument. But for \spinc manifolds there is an easier argument by combining the cohomological result $K_u^\ast(M) \barotimes \IC \cong \HudRco^\ast(M)$ with Theorem \ref{thm:local_thm_twisted} since taking the wedge product with $\ind(D)$ is an isomorphism on bounded de Rham cohomology, and furthermore using \Poincare duality between uniform $K$-theory and uniform $K$-homology (Theorem \ref{thm:Poincare_duality_K}), respectively between bounded de~Rham cohomology and uniform de~Rham homology (Theorem \ref{thm:Poincare_de_Rham}).

\section{Index theorems}\label{sec_index_theorems}

In this section we assemble everything that we had up to now into various index theorems. In Section~\ref{sec:local_index_thm} we first recall the construction of the topological index classes of elliptic operators and then prove local index theorems. In Section~\ref{sec_amenable_index} we prove a global index theorem, which will be a generalization of an index theorem of Roe \cite{roe_index_1}. He proved it for Dirac operators and we will generalize it to elliptic pseudodifferential operators.

\subsection{Local index formulas}
\label{sec:local_index_thm}

Let $M$ be a Riemannian manifold without boundary. We denote by $DM$ the disk bundle $\{\xi \in T^\ast M\colon \|\xi\|\le 1\}$ of its cotangent bundle and by $SM = \partial DM$ its boundary, i.e., $SM = \{\xi \in T^\ast M \colon \|\xi\| = 1\}$. If $M$ has bounded geometry, we may equip $DM$ with a Riemannian metric such that it also becomes of bounded geometry\footnote{Though we do not have defined bounded geometry for manifolds with boundary, there is an obvious one (demanding bounds not only for the curvature tensor of $M$ but also for the second fundamental form of the boundary of $M$, and demanding the injectivity radius being uniformly positive not only for $M$ but also for $\partial M$ with the induced metric). See \cite{schick_bounded_geometry_boundary} for a further discussion.} and $DM \to M$ becomes a Riemannian submersion. It follows that $SM$ will also have bounded geometry. What follows will be independent of the concrete choice of metric on $DM$. Though we have discussed in Section \ref{sec_uniform_K} only uniform $K$-theory for manifolds without boundary, one can of course define more generally relative uniform $K$-theory and discuss it for manifolds with boundary and of bounded geometry.

Let $P \in \UPsiDO^k(E)$ be a symmetric, elliptic and graded uniform pseudodifferential operator. Recall from Definition \ref{defnnkdf893} of ellipticity that the principal symbol $\sigma(P^+)$, viewed as a section of $\Hom(\pi^\ast E^+, \pi^\ast E^-) \to T^\ast M$, where $\pi\colon T^\ast M \to M$ is the cotangent bundle, is invertible outside a uniform neighbourhood of the zero section $M \subset T^\ast M$ and satisfies a certain uniformity condition. Then the well-known clutching construction gives us the following symbol class of $P$:
\[ \sigma_P := [\pi^\ast E^+, \pi^\ast E^-; \sigma(P) ] \in K_u^0(DM, SM).\]

If $P$ is ungraded, then its symbol $\sigma(P)\colon \pi^\ast E \to \pi^\ast E$, where $\pi\colon SM \to M$ denotes now the unit sphere bundle of $M$, is a uniform, self-adjoint automorphism. Hence it gives a direct sum decomposition $\pi^\ast E = E^+ \oplus E^-$, where $E^+$ and $E^-$ are spanned fiberwise by the eigenvectors belonging to the positive, respectively negative, eigenvalues of $\sigma(P)$, and we get an element
\[[E^+] \in K_u^0(SM).\]
Now we define in the ungraded case the symbol class of $P$ as
\[\sigma_P := \delta[E^+] \in K_u^1(DM,SM),\]
where $\delta\colon K_u^0(SM) \to K_u^1(DM,SM)$ is the boundary homomorphism of the $6$-term exact sequence associated to $(DM,SM)$. References for this construction in the compact case are, e.g., \cite[Section 24]{baum_douglas} and \cite[Proposition 3.1]{atiyah_patodi_singer_3}.

Applying the Chern character and integrating over the fibers we get in both the graded and ungraded case $\pi_! \ch \sigma_P \in \HbdR^\ast(M)$ and then the index class of $P$ is defined as
\[\ind(P) := (-1)^{\frac{n(n+1)}{2}} \pi_! \ch \sigma_P \wedge \operatorname{Td}(M) \in \HbdR^\ast(M),\]
where $n = \dim M$.

Let $M$ be a \spinc manifold of bounded geometry and let us denote by $D$ the Dirac operator associated to the \spinc structure of $M$. Note that it is $m$-multigraded, where $m$ is the dimension of the manifold $M$, and so defines an element in $K_m^u(M)$. Hence cap product with $D$ is a map $K_u^\ast(M) \to K^u_{m-\ast}(M)$, which is an isomorphism (Theorem~\ref{thm:Poincare_duality_K}). We have also the \Poincare duality map $\HbdR^\ast(M) \to \HudR_{m-\ast}(M)$, and the content of our local index theorem for uniform twisted Dirac operators is to put these duality maps into a commutative diagram using the homological Chern character on the right hand side and on the cohomology side the index class of the twisted operator.

\begin{thm}[Local index theorem for twisted uniform Dirac operators]\label{thm:local_thm_twisted}
Let $M$ be an $m$-dimensional \spinc manifold of bounded geometry and without boundary. Denote the associated Dirac operator by $D$.

Then we have the following commutative diagram:
\[\xymatrix{
K^\ast_u(M) \ar[r]^-{- \cap [D]}_-{\cong} \ar[d]_-{\ch(-) \wedge \ind(D)} & K_{m-\ast}^u(M) \ar[d]^-{\alpha_\ast \circ \ch^\ast}\\
\HbdR^\ast(M) \ar[r]_-{\cong} & \HudR_{m-\ast}(M)}\]
where in the top row $\ast$ is either $0$ or $1$ and in the bottom row $\ast$ is either $\ev$ or $\odd$.
\end{thm}

\begin{proof}
This follows from the calculations carried out by Connes and Moscovici in their paper \cite[Section 3]{connes_moscovici} by noting that the computations also apply in our case where we have bounded geometry and the uniformity conditions. Note that there the cyclic cocycles are defined using expressions in the operators $e^{-tD^2}$. To translate to the definition of the homological Chern character that we use, see, e.g., \cite[Section 10.2]{GBVF}.
\end{proof}

\begin{rem}\label{rem:ch_well_defined_spinc}
The uniform homological Chern character $\alpha_\ast \circ \ch^\ast\colon K_\ast^u(M) \dashrightarrow \HudR_\ast(M)$ is a priori not well-defined (to be more precise, it is defined on uniformly finitely summable Fredholm modules and it is a priori not clear whether it descends to classes and even whether every class may be represented by a uniformly finitely summable module). But using \Poincare duality between uniform $K$-homology and uniform $K$-theory and the above local index theorem, we see that it is a posteriori well-defined for \spinc manifolds. Note that since $D$ is a Dirac operator, it defines a uniformly finitely summable Fredholm module, and therefore also all its twists given by taking the cap product with uniform $K$-theory classes are uniformly finitely summable.

That the uniform homological Chern character is well-defined for every manifold $M$ of bounded geometry is content of Corollary \ref{cor:ch_well_defined}.
\end{rem}

Let $P$ be a symmetric and elliptic uniform pseudodifferential operator over an oriented manifold $M$ of bounded geometry. It defines a uniform $K$-homology class $[P] \in K_\ast^u(M)$ and therefore, if $P$ is in addition uniformly finitely summable, we may compare the class $(\alpha_\ast \circ \ch^\ast)(P) \in \HudR_\ast(M)$ with $\ind(P) \in \HbdR^\ast(M)$ using \Poincare duality. That they are equal is the content of the next theorem.

\begin{thm}[Local index formula for uniform pseudodifferential operators]\label{thm:local_thm_pseudos}
Let $M$ be an oriented Riemannian manifold of bounded geometry and without boundary.

Let $P$ be a symmetric and elliptic uniform pseudodifferential operator of positive order acting on a vector bundle $E \to M$ of bounded geometry, and let $P$ be uniformly finitely summable\footnote{This means that $P$ defines a uniformly finitely summable Fredholm module, i.e., $\chi(P)$ is uniformly finitely summable for some normalizing function $\chi$.}.

Then $\ind(P) \in \HbdR^\ast(M)$ is the \Poincare dual of $(\alpha_\ast \circ \ch^\ast)(P) \in \HudR_\ast(M)$.
\end{thm}

\begin{proof}
This follows from the above Theorem \ref{thm:local_thm_twisted} by the same arguments as in the proof of \cite[Theorem 3.9]{connes_moscovici}: if $M$ is odd-dimensional we take the product with $S^1$, and then we use the fact that for oriented, even-dimensional manifolds uniform $K$-homology is spanned modulo $2$-torsion by generalized signature operators. This last fact will follow from Theorem \ref{thm:thom} below.
\end{proof}

\begin{cor}\label{cor:ch_well_defined}
The uniform homological Chern character
\[\alpha_\ast \circ \ch^\ast \colon K_\ast^u(M) \to \HudR_\ast(M)\]
is well-defined for every manifold $M$ of bounded geometry and without boundary.
\end{cor}

\begin{proof}
If $M$ is \spinc we know by \Poincare duality that every class $[x] \in K_\ast^u(M)$ may be represented by a uniformly finitely summable Fredholm module and by the above Theorem \ref{thm:local_thm_pseudos} we conclude that $(\alpha_\ast \circ \ch^\ast)([x])$ is independent of the concrete choice of such a representative. (This was already mentioned in Remark \ref{rem:ch_well_defined_spinc}.)

In the general case we first pass to the orientation cover $X$ if $M$ is not orientable. Note that if we know the statement that we want to prove for a finite covering of $M$, then we know it also for $M$ itself since $\HudR_\ast(M)$ is a vector space over $\IC$ (i.e., multiplication by some non-zero number is an isomorphism). Now we can go on as in the proof of Theorem \ref{thm:local_thm_pseudos}: we take the product with $S^1$ if necessary and then use the fact that on oriented, even-dimensional manifolds we can represent every uniform $K$-homology class by a multiple (concretely, $2^{\dim(M)/2}$) of a generalized signature operator. For the latter statement see Theorem \ref{thm:thom}, respectively its proof.
\end{proof}

\begin{rem}\label{rem_finitely_summable_not_needed}
The condition in the above Theorem \ref{thm:local_thm_pseudos} that the operator $P$ is uniformly finitely summable may be dropped. The statement then is that $(\alpha_\ast \circ \ch^\ast)([P])$ is the dual of the class $\ind(P) \in \HbdR^\ast(M)$. This makes sense since we now know that the uniform homological Chern character $K_\ast^u(M) \to \HudR_\ast(M)$ is well-defined.

But the problem then is that in order to compute $(\alpha_\ast \circ \ch^\ast)([P])$ we would have to replace $P$ by some other operator $P^\prime$ which defines the same uniform $K$-homology class as $P$ but which is uniformly finitely summable (so that we may compute the Chern--Connes character). This seems to be a task which is not easily carried out in practice.

Connes and Moscovici work in \cite{connes_moscovici} with so-called \emph{$\theta$-summable Fredholm modules} which are more general than finitely summable modules. So defining an appropriate version of \emph{uniformly $\theta$-summable Fredholm modules} we could certainly prove the above Theorem \ref{thm:local_thm_pseudos} for them and therefore weakening the condition on $P$ that it has to be uniformly finitely summable.
\end{rem}

Let us state now the Thom isomorphism theorem in the form that we need for the proof of the above Theorem \ref{thm:local_thm_pseudos}.

\begin{thm}[Thom isomorphism]\label{thm:thom}
Let $M$ be a Riemannian \spinc manifold of bounded geometry and without boundary.

Then the principal symbol of the Dirac operator associated to the \spinc structure of $M$ constitutes an orientation class in $K_u^\ast(DM, SM)$, i.e., it implements the isomorphism $K_u^\ast(M) \cong K_u^\ast(DM, SM)$.

If $M$ is only oriented (i.e., not necessarily spin$^c$) and even-dimensional, the principal symbol of the signature operator of $M$ constitutes an orientation class in $K_u^\ast(DM, SM)[\tfrac{1}{2}]$.
\end{thm}

\begin{proof}
The usual proof as found in, e.g., \cite[Appendix C]{lawson_michelsohn}, works in our case analogously. Note that for the proof of \cite[Theorem C.7]{lawson_michelsohn} we have to cover $M$ by such subsets as we used in our proof of \Poincare duality (see Lemma \ref{lem:suitable_coloring_cover_M}) since only in this case we have shown that we have a Mayer--Vietoris sequence for uniform $K$-theory. For the statement for only oriented $M$ see, e.g., the proof of \cite[Theorem C.12]{lawson_michelsohn}.
\end{proof}

In \cite[Theorem 3.9]{connes_moscovici} the local index theorem was written using an index pairing with compactly supported cohomology classes. We can of course do the same also here in our uniform setting and the statement is at first glance the same.\footnote{Remember that we have another choice of universal constants than Connes and Moscovici, i.e., in our statement they are not written since they are incorporated in the definition of the homological Chern character.} But the difference is that due to the uniformness we have an additional continuity statement.

\begin{cor}\label{cor:pairing_compactly_supported}
Let $[\varphi] \in H_{c, \mathrm{dR}}^k(M)$ be a compactly supported cohomology class and define the analytic index $\ind_{[\varphi]}(P)$ as in \cite{connes_moscovici}.\footnote{Note that $\ind_{[\varphi]}(P)$ is analytically defined and may be computed (up to the universal constant that we have incorporated into the definition of $\alpha_\ast \circ \ch^\ast$) as $\langle (\alpha_\ast \circ \ch^\ast)(P), [\varphi] \rangle$, where $\langle -,- \rangle$ is the pairing between uniform de Rham homology and compact supported cohomology.} Then we have
\[\ind_{[\varphi]}(P) = \int_M \ind(P) \wedge [\varphi]\]
and this pairing is continuous, i.e., $\int_M \ind(P) \wedge [\varphi] \le \|\ind(P) \|_\infty \cdot \| [\varphi] \|_1$, where $\| - \|_\infty$ denotes the sup-seminorm on $\HbdR^{m-k}(M)$ and $\| - \|_1$ the $L^1$-seminorm on $H_{c, \mathrm{dR}}^k(M)$.
\end{cor}

\begin{proof}
The corollary follows from Theorem \ref{thm:local_thm_pseudos} (if $M$ is not orientable then we first have to pass to the orientation cover of it). The continuity statement follows from the definition of the seminorms. The only thing we have to know is that $\ind(P)$ is given by a bounded de Rham form.
\end{proof}

\begin{rem}\label{rem:local_pairing_cont}
Though it may seem that the above corollary is in some sense equivalent to Theorem \ref{thm:local_thm_pseudos}, it is in fact not. It is weaker in the following way: in case of a non-compact manifold $M$ the bounded de Rham cohomology $\HbdR^\ast(M)$ usually contains elements of seminorm $=0$ and due to the boundedness of the above pairing we see that we can not detect these elements by it.
\end{rem}

\subsection{Index pairings on amenable manifolds}
\label{sec_amenable_index}

In the last section we proved the local index theorems for uniform operators. The goal of this section is to use these local formulas to compute certain global indices of such operators over amenable manifolds.

So in this section we assume that our manifold $M$ is \emph{amenable}, i.e., that it admits a \Folner sequence. We will need such a sequence in order to construct the index pairings.

\begin{defn}[\Folner sequences]
Let $M$ be a manifold of bounded geometry. A sequence of compact subsets $(M_i)_i$ of $M$ will be called a \emph{F{\o}lner sequence}\footnote{In \cite[Definition 6.1]{roe_index_1} such sequences were called \emph{regular}.} if for each $r > 0$ we have
\[\frac{\vol B_r(\partial M_i)}{\vol M_i} \stackrel{i \to \infty}\longrightarrow 0.\]

A F{\o}lner sequence $(M_i)_i$ will be called a \emph{F{\o}lner exhaustion}, if $(M_i)_i$ is an exhaustion, i.e., $M_1 \subset M_2 \subset \ldots$ and $\bigcup_i M_i = M$.
\end{defn}

Note that if $M$ admits a \Folner sequence, then it is always possible to construct a \Folner exhaustion for $M$ (the author did this construction in its full glory in his thesis \cite[Lemma 2.38]{engel_phd}).

For example, Euclidean space $\IR^m$ is amenable, but hyperbolic space $\mathbb{H}^{m \ge 2}$ is not. Furthermore, if $M$ has subexponential volume growth at $x_0 \in M$,\footnote{This means that for all $p > 0$ we have $e^{-pr} \vol(B_r(x_0)) \xrightarrow{r \to \infty} 0$.} then $M$ is amenable (this is proved in \cite[Proposition 6.2]{roe_index_1}; in this case a \Folner exhaustion for $M$ is given by $\big(B_{r_j}(x_0)\big)_{j \in \IN}$ for suitable $r_j \to \infty$). Note that the converse to this last statement is wrong, i.e., there are examples of amenable spaces with exponential volume growth. Further examples of amenable manifolds arise from the theorem that the universal covering $\widetilde{M}$ of a compact manifold $M$ is amenable (if equipped with the pull-back metric) if and only if the fundamental group $\pi_1(M)$ is amenable (this is proved in \cite{brooks}).

Let $M^m$ be a connected and oriented manifold of bounded geometry. Then there is a duality isomorphism $H_{b, \mathrm{dR}}^m(M) \cong H_0^{\mathrm{uf}}(M; \IR)$, where the latter denotes the uniformly finite homology of Block and Weinberger. This isomorphism is mentioned in the remark at the end of Section 3 in \cite{block_weinberger_1} and proved explicitely in \cite[Lemma 2.2]{whyte}.\footnote{Alternatively, we could use the \Poincare duality isomorphism $H_{b, \mathrm{dR}}^i(M) \cong H_{m-i}^\infty(M; \IR)$ which is proved in \cite[Theorem 4]{attie_block_1}, where $H_{m-i}^\infty(M; \IR)$ denotes simplicial $L^\infty$-homology and $M$ is triangulated according to Theorem \ref{thm:triangulation_bounded_geometry}, and then use the fact that $H_0^\infty(M; \IR) \cong H_0^{\mathrm{uf}}(M; \IR)$ under this triangulation (for this we need the assumption that $M$ is connected).} Since we have the characterization \cite[Theorem 3.1]{block_weinberger_1} of amenability stating that $M$ is amenable if and only if $H_0^{\mathrm{uf}}(M) \not= 0$, we therefore also have a characterization of it via bounded de Rham cohomology. We are going to discuss this now a bit more closely.

First we introduce the following notions:

\begin{defn}[Closed at infinity, {\cite[Definition II.5]{sullivan}}]
A Riemannian manifold $M$ is called \emph{closed at infinity} if for every function $f$ on $M$ with $0 < C^{-1} < f < C$ for some $C > 0$, we have $[f \cdot dM] \not= 0 \in H_{b, \mathrm{dR}}^m(M)$ (where $dM$ denotes the volume form of $M$ and $m = \dim M$).
\end{defn}

\begin{defn}[Fundamental classes, {\cite[Definition 3.3]{roe_index_1}}]\label{defn:fundamental_class}
A \emph{fundamental class} for the manifold $M$ is a positive linear functional $\theta\colon \Omega^m_b(M) \to \IR$ such that $\theta(dM) \not = 0$ and $\theta \circ d = 0$.
\end{defn}

If we are given a F{\o}lner sequence for $M$, we can construct a fundamental class for $M$ out of it; this is done in \cite[Propositions 6.4 \& 6.5]{roe_index_1}.\footnote{If $(M_i)_i$ is a \Folner sequence, then the linear functionals $\theta_i(\alpha) := \frac{1}{\vol M_i} \int_{M_i} \alpha$ are elements of the dual of $\Omega_b^m(M)$ and have operator norm $= 1$. Now take $\theta$ as a weak-$^\ast$ limit point of $(\theta_i)_i$. The \Folner condition for $(M_i)_i$ is needed to show that $\theta$ vanishes on boundaries.} But admitting a fundamental class implies that $M$ is closed at infinity.\footnote{Just use the positivity of the fundamental class $\theta$: $\theta(f \cdot dM) \ge \theta(C^{-1} \cdot dM) = C^{-1} \cdot \theta(dM) \not= 0$.} This means especially $H_{b, \mathrm{dR}}^m(M) \not= 0$. But since this is isomorphic to $H_0^{\mathrm{uf}}(M; \IR)$, we conclude that the latter does also not vanish. So $M$ is amenable, i.e., admits a \Folner sequence, and so we are back at the beginning of our chain. Let us summarize this:

\begin{prop}
Let $M$ be a connected, orientable manifold of bounded geometry.

Then the following are equivalent:
\begin{itemize}
\item $M$ admits a F{\o}lner sequence,
\item $M$ admits a fundamental class and
\item $M$ is closed at infinity.
\end{itemize}
\end{prop}

We know that the universal cover $\widetilde{M}$ of a compact manifold $M$ is amenable if and only if $\pi_1(M)$ is amenable. If this is the case, then we may construct fundamental classes that respect the structure of $\widetilde{M}$ as a covering space:

\begin{prop}[{\cite[Proposition 6.6]{roe_index_1}}]\label{prop:fundamental_group_amenable_nice_fundamental_classes}
Let $M$ be a compact Riemannian manifold, denote by $\widetilde{M}$ its universal cover equipped with the pull-back metric, and let $\pi_1(M)$ be amenable.

Then $\widetilde{M}$ admits a fundamental class $\theta$ with the property
\[\theta(\pi^\ast \alpha) = \int_M \alpha\]
for every top-dimensional form $\alpha$ on $M$ and where $\pi \colon \widetilde{M} \to M$ is the covering projection.
\end{prop}

At last, let us state just for the sake of completeness the relation of amenability to the linear isoparametric inequality.

\begin{prop}[{\cite[Subsection 4.1]{gromov_hyperbolic_manifolds_groups_actions}}]
Let $M$ be a connected and orientable manifold of bounded geometry.

Then $M$ is not amenable if and only if $\vol(R) \le C \cdot \vol(\partial R)$ for all $R \subset M$ and a fixed constant $C > 0$.
\end{prop}

We can also detect amenability of $M$ using the $K$-theory of the uniform Roe algebra $C_u^\ast(\Gamma)$ of a discretization $\Gamma \subset M$.\footnote{A discretization $\Gamma \subset M$ is a uniformly discrete subset such that there exists a $c > 0$ with $N_c(\Gamma) = M$, where $N_c(\Gamma)$ denotes the neighbourhood of distance $c$ around $\Gamma$.} Recall that one possible definition for the uniform Roe algebra $C_u^\ast(\Gamma)$ is the norm closure of the $^\ast$-algebra of all finite propagation operators in $\IB(\ell^2(\Gamma))$ with uniformly bounded coefficients.

\begin{prop}[{\cite{elek}}]
Let $M$ be a manifold of bounded geometry and let $\Gamma \subset M$ be a discretization.

Then $M$ is amenable if and only if $[1] \not= [0] \in K_0(C_u^\ast(\Gamma))$, where $[1] \in K_0(C_u^\ast(\Gamma))$ is a certain distinguished class.
\end{prop}

The reason why we stated the above proposition is that it introduces functionals on $K_0(C_u^\ast(\Gamma))$ associated to \Folner sequences that we will need in the definition of our index pairings. So let us recall Elek's argument: Let $(\Gamma_i)_i$ be a \Folner sequence in $\Gamma$\footnote{This means that each $\Gamma_i$ is finite and for every $r > 0$ we have $\frac{\card \partial_r \Gamma_i}{\card \Gamma_i} \xrightarrow{i \to \infty} 0$, where\[\partial_r \Gamma_i := \{\gamma \in \Gamma\colon d(\gamma,\Gamma_i) < r\text{ and }d(\gamma,\Gamma-\Gamma_i)< r\}\]and the distance is computed in $M$ (which makes sense since $\Gamma \subset M$).} and let $T \in C_u^\ast(\Gamma)$. Then we define a bounded sequence indexed by $i$ by $\frac{1}{\card \Gamma_i} \sum_{\gamma \in \Gamma_i} T(\gamma, \gamma)$. Choosing a linear functional $\tau \in (\ell^\infty)^\ast$ associated to a free ultrafilter on $\IN$\footnote{That is, if we evaluate $\tau$ on a bounded sequence, we get the limit of some convergent subsequence.} we get a linear functional $\theta$ on $C_u^\ast(\Gamma)$. The \Folner condition for $(\Gamma_i)_i$ is needed to show that $\theta$ is a trace, i.e., descends to $K_0(C_u^\ast(\Gamma))$. Then $\theta([1]) = 1$ and $\theta([0]) = 0$ for the distinguished classes $[1], [0] \in K_0(C_u^\ast(\Gamma))$.

Let us finally come to the definition of the index pairings that we are interested in.

\begin{defn}
Let $M$ be a manifold of bounded geometry, $(M_i)_i$ a \Folner sequence for $M$ and let $\tau \in (\ell^\infty)^\ast$ a linear functional associated to a free ultrafilter on $\IN$. Denote the resulting functional on $K_0(C_u^\ast(\Gamma))$ by $\theta$, where $\Gamma \subset M$ is a discretization.\footnote{Note that here we first have to construct from the \Folner sequence $(M_i)_i$ for $M$ a corresponding \Folner sequence $(\Gamma_i)_i$ for $\Gamma$.}

Then we define for $p=0,1$ an index pairing
\[\langle -,- \rangle_\theta \colon K^p_u(M) \otimes K_p^u(M) \to \IR\]
by the formula
\[\langle [x], [y] \rangle_\theta := \theta \big( \mu_u([x] \cap [y]) \big),\]
where $\mu_u\colon K_\ast^u(M) \to K_\ast(C_u^\ast(\Gamma))$ denotes the rough assembly map (see \Spakula \cite{spakula_uniform_k_homology} or \cite[Section 3.5]{engel_indices_UPDO}).
\end{defn}

If $P$ is a symmetric and elliptic, graded uniform pseudodifferential operator acting on a graded vector bundle $E$, then there is a nice way of computing the above index pairing of $P$ with the trivial bundle $[\IC] \in K^0_u(M)$: recall from Corollary \ref{cor:schwartz_function_of_PDO_quasilocal_smoothing} that if $f \in \mathcal{S}(\IR)$ is a Schwartz function, then $f(P)$ is a quasilocal smoothing operator. Hence it has a uniformly bounded integral kernel $k_{f(P)}(x,y) \in C_b^\infty(E \boxtimes E^\ast)$. Now we choose an even function $f \in \mathcal{S}(\IR)$ with $f(0) = 1$ and get a bounded sequence
\[\frac{1}{\vol M_i} \int_{M_i} \trace_s k_{f(P)}(x,x) \ dM(x),\]
where $\trace_s$ denotes the super trace (recall that $E$ is graded), on which we may evaluate $\tau$. This will coincide with the pairing $\langle [\IC], P \rangle_\theta$ and is exactly the analytic index that was defined by Roe in \cite{roe_index_1} for Dirac operators. For details why this will coincide with $\langle [\IC], P \rangle_\theta$ the reader may consult, e.g., the author's Ph.D.\ thesis \cite[Section 2.8]{engel_phd}.

Let us now define the pairing between uniform de Rham cohomology and uniform de Rham homology. So let $\beta \in C_b^\infty(\Omega^p(M))$ and $C \in \Omega_p^u(M)$, fix an $\epsilon > 0$ and choose for every $M_i \subset M$ from a \Folner sequence for $M$ a smooth cut-off function $\varphi_i \in C_c^\infty(M)$ with $\varphi_i|_{M_i} \equiv 1$, $\supp \varphi_i \subset B_\epsilon(M_i)$ and such that for all $k \in \IN_0$ the derivatives $\nabla^k \varphi_i$ are bounded in sup-norm uniformly in the index $i$. Then $\varphi_i \beta \in W^{\infty, 1}(\Omega^p(M))$ and therefore we may evaluate $C$ on it. The sequence $\frac{1}{\vol M_i} C(\varphi_i \beta)$ will be bounded and so we may apply $\tau \in (\ell^\infty)^\ast$ to it. Due to the \Folner condition for $(M_i)_i$ this pairing will descend to (co-)homology classes.

\begin{defn}
Let $M$ be a manifold of bounded geometry, let $(M_i)_i$ be a \Folner sequence for $M$ and let $\tau \in (\ell^\infty)^\ast$ a linear functional associated to a free ultrafilter on $\IN$.

For every $p \in \IN_0$ we define a pairing
\[\langle -,-\rangle_{(M_i)_i, \tau}\colon \HudRco^p(M) \otimes \HudR_p(M) \to \IC\]
by evaluating $\tau$ on the sequence $\frac{1}{\vol M_i} C(\varphi_i \beta)$, where $\beta \in \HudRco^p(M)$, $C \in \HudR_p(M)$ and the cut-off functions $\varphi_i$ are chosen as above.
\end{defn}

Note that this pairing is, similar to the pairing from Corollary \ref{cor:pairing_compactly_supported}, continuous against the topologies on $\HudRco^\ast(M)$ and on $\HudR_\ast(M)$.

Recall that in the usual case of compact manifolds the index pairing for $K$-theory and $K$-homology is compatible with the Chern-Connes character, i.e., $\langle [x], [y] \rangle = \langle \ch ([x]), \ch ([y]) \rangle$ for $[x] \in K^\ast(M)$ and $[y] \in K_\ast(M)$. The same also holds in our case here.

\begin{lem}
Denote by $\ch\colon K_u^\ast(M) \to \HudRco^\ast(M)$ the Chern character on uniform $K$-theory and by $(\alpha_\ast \circ \ch^\ast)\colon K_\ast^u(M) \to \HudR_\ast(M)$ the one on uniform $K$-homology.

Then we have
\[\big\langle [x], [y] \big\rangle_\theta = \big\langle\ch([x]), (\alpha_\ast \circ \ch^\ast)([y]) \big\rangle_{(M_i)_i, \tau}\]
for all $[x] \in K_u^p(M)$ and $[y] \in K^u_p(M)$.
\end{lem}

The last thing that we need is the compatibility of the index pairings with cup and cap products. This is clear by definition for the index pairing for uniform $K$-theory with uniform $K$-homology, and for the pairing for uniform de Rham cohomology with uniform de Rham homology it is stated in the following lemma.

\begin{lem}
Let $[\beta] \in \HudRco^p(M)$, $[\gamma] \in \HudRco^q(M)$ and $[C] \in \HudR_{p+q}(M)$. Then we have
\[\langle [\beta] \wedge [\gamma], [C] \rangle_{(M_i)_i, \tau} = \langle [\beta], [\gamma] \cap [C] \rangle_{(M_i)_i, \tau}.\]
\end{lem}

So combining the above two lemmas together with the results of Section \ref{sec:local_index_thm} we finally arrive at our desired index theorem for amenable manifolds which generalizes Roe's index theorem from \cite{roe_index_1} from graded generalized Dirac operators to arbitrarily graded, symmetric, elliptic uniform pseudodifferential operators.

\begin{cor}\label{cor:pairing_global}
Let $M$ be a manifold of bounded geometry and without boundary, let $(M_i)_i$ be a \Folner sequence for $M$ and let $\tau \in (\ell^\infty)^\ast$ be a linear functional associated to a free ultrafilter on $\IN$. Denote the from the choice of \Folner sequence and functional $\tau$ resulting functional on $K_0(C_u^\ast(\Gamma))$ by $\theta$, where $\Gamma \subset M$ is a discretization.

Then for both $p \in \{ 0,1 \}$, every $[P] \in K_p^u(M)$ for $P$ a $p$-graded, symmetric, elliptic uniform pseudodifferential operator over $M$, and every $u \in K_u^p(M)$ we have
\[\langle u, [P] \rangle_\theta = \langle \ch(u) \wedge \ind(P), [M] \rangle_{(M_i)_i, \tau}.\]
\end{cor}

\begin{rem}
The right hand side of the formula in the above corollary reads as
\[\tau \Big( \frac{1}{\vol M_i} \int_{M_i} \ch(u) \wedge \ind(P) \Big)\]
and this is continuous against the sup-seminorm on $\HbdR^m(M)$ with $m = \dim(M)$, i.e.,
\[\langle u, [P] \rangle_\theta \le \| \ch(u) \wedge \ind(P) \|_\infty.\]
So, again as in Remark \ref{rem:local_pairing_cont}, we see that with this pairing we can not detect operators $P$ whose index class $\ind(P) \in \HbdR^\ast(M)$ has sup-seminorm $=0$ in every degree.

Note that it seems that from the results in \cite[Part II.\S 4]{sullivan} it follows that every element in $\HbdR^m(M)$ of non-zero sup-seminorm may be detected by a \Folner sequence (i.e., the dual space $\overline{H}_{b, \mathrm{dR}}^\ast(M)$ of the reduced bounded de Rham cohomology\footnote{Reduced bounded de Rham cohomology is defined as $\overline{H}_{b, \mathrm{dR}}^\ast(M) := \HbdR^\ast(M) / \overline{[0]}$, i.e., as the Hausdorffication of bounded de Rham cohomology.} is spanned by \Folner sequences). So the difference between the statement of the above corollary and Theorem \ref{thm:local_thm_pseudos} lies, at least in top-degree, exactly in the fact that Theorem \ref{thm:local_thm_pseudos} also encompasses all the elements of sup-seminorm $=0$.
\end{rem}

\begin{example}
Let us discuss quickly an example that shows that we indeed may lose information by passing to the reduced bounded de Rham cohomology groups. Roe showed in \cite[Proposition 3.2]{roe_index_2} that if $M^m$ is a connected spin manifold of bounded geometry, then $\langle \hat{A}(M), [M] \rangle_{-,-} = 0$ for any choice of \Folner sequence and suitable functional $\tau$ if $M$ has non-negative scalar curvature, and later Whyte showed in \cite[Theorem 2.3]{whyte} that $\hat{A}(M) = [0] \in \HbdR^m(M)$ under these assumptions. So any connected spin manifold $M$ of bounded geometry with $\hat{A}(M) \not= [0] \in \HbdR^m(M)$ but $\hat{A}(M) = [0] \in \overline{H}_{b, \mathrm{dR}}^m(M)$ can not have non-negative scalar curvature, but this is not detected by the reduced group. In \cite{whyte} it is also shown how one can construct examples of manifolds whose $\hat{A}$-genus vanishes in the reduced but not in the unreduced group.
\end{example}

\section{Final remarks and open questions}

In this final section we will collect some open questions arising out of the present paper.

In Theorem~\ref{thm_Chern_character} we constructed the uniform Chern character $\ch \colon K^\ast_u(M) \to \HudRco^\ast(M)$ for manifolds of bounded geometry by using Chern--Weil theory. In the case of compact spaces there exist definitions of the Chern character which make sense on any finite CW-complex, i.e., are not restricted to smooth manifolds. Now in our situation, uniform $K$-theory is defined on all metric spaces and not just on manifolds, and since uniform de Rham cohomology is isomorphic to $L^\infty$-simplicial cohomology if we triangulate $M$ as a simplicial complex of bounded geometry using Theorem \ref{thm:triangulation_bounded_geometry} we can make sense out of it for more general spaces than smooth manifolds. So we arive at the following question:

\begin{question}\label{ques:uniform_chern}
How can we define the uniform Chern characters $K_u^\ast(L) \to H^\ast_\infty(L)$ and $K_\ast^u(L) \to H_\ast^\infty(L)$ for a simplicial complex $L$ of bounded geometry equipped with the metric derived from barycentric coordinates?
\end{question}

One approach might be to consider something like uniform (co-)homology theories: we could try to put a model structure on the category of uniform spaces modeling uniform homotopy theory and then try to show that, e.g., uniform $K$-theory is nothing more but uniform homotopy classes of uniform maps into some uniform version of the $K$-theory spectrum. Then the uniform Chern characters should be coming from transformations of uniform spectra and the above Question \ref{ques:uniform_chern} would be solved.

In the compact case there is a generalization of the Atiyah--Singer index theorem to manifolds with boundary involving the $\eta$-invariant. This version of the index theorem for compact manifolds with boundary is called the Atiyah--Patodi--Singer index theorem and was introduced in \cite{atiyah_patodi_singer_1}. Of course the question whether such a theorem may also be proven in the non-compact case immediately arises.

\begin{question}
Is there a version of the, e.g., global index theorem for amenable manifolds, for manifolds of bounded geometry and with boundary? What would be the corresponding generalization of the $\eta$-invariant?
\end{question}

Note that even if we just stick to Dirac operators (i.e., if we don't try to work with uniform pseudodifferential operators) the non-compact case (of bounded geometry) is of course technically much more demanding than the compact case. Results have been achieved by Ballmann--Bär \cite{ballmann_baer} and Gro{\ss}e--Nakad \cite{grosse_nakad}.

Some version of pseudodifferential operators on certain non-compact manifolds with boundary was investigated by Schrohe \cite{schrohe}. Furthermore, there is also the work of Ammann--Lauter--Nistor \cite{ammann_lauter_nistor_2} and one should also ask to which extend it coincides, respectively differs from the one asked for here.

A proof of the index theorem for manifolds with boundary was given by Melrose in \cite{melrose_APS}. He invented the $b$-calculus, a calculus for pseudodifferential operators on manifolds with boundary, and derived the Atiyah--Patodi--Singer index theorem from it via the heat kernel approach. Therefore it would be desirable to extend his $b$-calculus to open manifolds with boundary (similarly as we extended the calculus of pseudodifferential operators to open manifolds) and then prove a version of the Atiyah--Patodi--Singer index theorem on manifolds with boundary and of bounded geometry.

\begin{question}
Can one reasonably extend the $b$-calculus of Melrose to manifolds of bounded geometry and with boundary, and then prove version of large scale index theorems for manifolds with boundary?
\end{question}

In the case of compact manifolds with boundary Piazza \cite{piazza_phd} also treated various parts of the index theorem of Atiyah--Patodi--Singer using the $b$-calculus. A connection between uniform pseudodifferential operators on manifolds of bounded geometry and the $b$-calculus was established by Albin \cite{albin}.

Another direction in which one could work is to look at higher $\rho$-invariants: in the last years a lot of progress was made in relation to ``mapping sugery to analysis'', respectively mapping the Stolz positive scalar curvature exact sequence to analysis. Without going through all the results that have been achieved, let us mention one particular application \cite[Corollary 4.5]{xie_yu} that seems worth reshaping into our setting: if $\Gamma$ acts properly and cocompactly on $M$ and $h$ is a Riemannian metric on $\partial M$ having positive scalar curvature, then it is not possible to extend $h$ to a complete, $\Gamma$-invariant Riemannian metric on~$M$ of positive scalar curvature if $\rho(D_{\partial M},h) \not= 0 \in K_\ast(C^\ast_{L,0}(\partial M)^\Gamma)$.

\begin{question}
Can one prove a large scale version of the delocalized APS-index theorem as in \cite[Theorem 1.22]{piazza_schick} and use this to prove an analogue of the above mentioned result \cite[Corollary 4.5]{xie_yu}?
\end{question}

\bibliography{./Bibliography_Index_theory_uniformly_elliptic}
\bibliographystyle{amsalpha}

\end{document}